\documentclass[12pt,a4paper]{amsart}

\usepackage{amsmath,amsthm,amssymb,latexsym,a4wide,tikz,multicol,tikz-cd}
\usepackage{arydshln,multirow}
\usepackage{tikz-qtree,tikz-qtree-compat}
\usepackage{mathtools,stmaryrd}
\usepackage{comment}
\usepackage{tikz}
\tikzset{font=\small}
\usepackage{enumerate}
\usetikzlibrary{matrix,arrows}
\usetikzlibrary{positioning}
\usetikzlibrary{cd}

\newtheorem{theorem}{Theorem} [section]
\newtheorem{lemma}[theorem]{Lemma}
\newtheorem{corollary}[theorem]{Corollary}
\newtheorem{proposition}[theorem]{Proposition}

\theoremstyle{definition}
\newtheorem{remark}[theorem]{Remark}
\newtheorem{definition}[theorem]{Definition}
\newtheorem{question}[theorem]{Question}
\newcommand{\ran}{\mathrm{ran}}
\newcommand{\dom}{\mathrm{dom}}
\newcommand{\vertleq}{\rotatebox{90}{$\leqslant$}}
\newcommand{\vertgeq}{\rotatebox{90}{$\geqslant$}}



\usepackage[active]{srcltx}

\usepackage{enumerate}
\numberwithin{equation}{section}

\title{Proper Ehresmann semigroups}

\author{Ganna Kudryavtseva}
\address{G. Kudryavtseva: University of Ljubljana,
Faculty of Mathematics and Physics, Jadranska ulica~19, SI-1000 Ljubljana, Slovenia and Institute of Mathematics, Physics and Mechanics, Jadranska ulica 19, SI-1000 Ljubljana, Slovenia}
\email{ganna.kudryavtseva\symbol{64}fmf.uni-lj.si}
\author{Valdis Laan}
\address{V. Laan: Institute of Mathematics and Statistics, University of Tartu,
51009 Tartu, Estonia}
\email{valdis.laan\symbol{64}ut.ee}

\thanks{This work was supported by the Slovenian Research Agency grant BI-EE/20-22-010. Research of the first named author was supported by  the Slovenian Research Agency grant P1-0288. Research of the second named author was supported by the Estonian Research Council grant PRG1204.}


\begin{document}
\begin{abstract}  We propose a notion of a proper Ehresmann semigroup based on a three-coordinate description of its generating elements governed by certain labelled directed graphs with additional structure. The generating elements are  determined by their domain projection, range projection and 
$\sigma$-class, where $\sigma$ denotes the minimum congruence that identifies all projections. We prove a structure result on proper Ehresmann semigroups and show that every Ehresmann semigroup has a proper cover. Our covering monoid turns out to be isomorphic to that from the work by Branco, Gomes and Gould and provides a new view of the latter. Proper Ehresmann semigroups all of whose elements admit a three-coordinate description are characterized in terms of partial multiactions of monoids  on semilattices. As a consequence we recover the two-coordinate structure result on proper restriction semigroups.
\end{abstract}

\maketitle

\section{Introduction}\label{s:introduction}
Ehresmann semigroups and their one-sided analogues are widely studied non-regular generalizations of inverse semigroups, see, e.g., \cite{BGG11, EG21, GG01, GG11, Kam11, KK14, KL17, Lawson91, L21,MS21,S17,S18}. They possess two unary operations $a\mapsto a^+$ and $a\mapsto a^*$ which mimic the operations of taking the domain idempotent and the range idempotent in an inverse semigroup. Ehresmann semigroups were defined in the paper by Lawson \cite{Lawson91} and arise naturally from certain categories which appear in the work of the school of Charles Ehresmann on differential geometry.  Lawson's results \cite{Lawson91} generalize the famous Ehresmann-Schein-Nambooripad theorem \cite[Theorem 8.4.1]{Lawson} that connects inverse semigroups with inductive groupoids, as well as its extension that connects restriction semigroups with inductive categories \cite{A84}.  

Ehresmann semigroups looked as algebras $(S; \cdot, ^+, ^*)$ form a variety of algebras. The elements of free objects of this variety admit an elegant graphical description in terms of birooted labelled directed trees, found by Kambites in \cite{Kam11}. This description extends the description of elements of the free restriction semigroups \cite{FGG09, K19} by means of the Munn's famous construction of elements of the free inverse semigroups \cite{Munn74}.

Perhaps the most natural examples of Ehresmann semigroups are the semigroup ${\mathcal{B}}(X)$ of all binary relations on a set $X$ and its subsemigroup ${\mathcal{PT}}(X)$ of all partially defined self-maps of $X$. The latter semigroup is an Ehresmann semigroup which is in addition left restriction, that is, it satisfies an additional identity called the left ample identity (see Section \ref{s:prelim} for the  definition). Extending the results obtained for inverse semigroups by Steinberg \cite{St06,St08},   Stein \cite{S17, S18} proved that the semigroup algebra of a left restriction Ehresmann semigroup is isomorphic to the semigroup algebra of its attached Ehresmann category and applied this result to studying representations of left restriction Ehresmann semigroups. This study was further pursued by Margolis and Stein in \cite{MS21}. 

Ehresmann semigroups appear naturally in the work on non-commutative Stone duality by Lawson and the first named author \cite{KL17}. They are constructed from certain, localic or topological, \'etale categories and possess the additional structure of quantales, see also the recent work by Lawson \cite{L21} where a construction of Ehresmann semigroups from categories inspired by \cite{KL17} is extensively studied.

East and Gray \cite{EG21} have recently studied Ehresmann structures arising on partition monoids and related diagram monoids such as Brauer monoids and rook partition monoids.

The purpose of the present paper is to contribute to the development of the appropriate notion of a proper Ehresmann semigroup, initiated by Branco, Gomes and Gould in \cite{BGG15}, see also Branco et al. \cite{BGGW18}.  An Ehresmann monoid $S$ is defined in \cite{BGG15} to be {\em strongly} $T$-{\em proper},  where $T$ is a monoid, if $T$ is contained in $S$ as a submonoid, $S$ is generated by $T$ and the projections of $S$, and the congruence $\sigma$ (which is the least congruence that identifies all projections) separates $T$. It is proved in \cite{BGG15} that the free Ehresmann monoid $FEM(X)$ is $T$-proper (for $T$ being isomorphic to the free monoid $X^*$) and that every Ehresmann monoid has a $T$-proper cover (that is, for every Ehresmann monoid $S$ there is some $T$-proper Ehresmann monoid $P$ and a surjective projection-separating morphism of Ehresmann monoids from $P$ onto $S$). It has been observed by Jones \cite{Jones16} that, specialized to the class of restriction monoids, $T$-properness reduces to the property of being a perfect restriction monoid  \cite{Jones16, K15}, thus $T$-proper Ehresmann monoids form a natural class of Ehresmann monoids. However, since perfect restriction monoids do not exhaust all proper restriction monoids, $T$-proper Ehresmann monoids generalize a subclass of proper restriction monoids, rather than the whole class of proper restriction monoids. Our motivaton for this research was to develop a notion of a proper Ehresmann semigroup which would not necessarily be a monoid and which would generalize the notion of a proper restriction semigroup. 

Our approach is new, and is not based on that adopted in \cite{BGG15,BGGW18}: our notion of a proper Ehresmann semigroup relies on matching factorizatons of elements into products of certain generating elements, which can be interpreted as arrows of a labelled directed graph (see Section \ref{s:structure} for details), whereas the $T$-proper Ehresmann monoid ${\mathcal P}(T,Y)$ of \cite{BGG15,BGGW18} is defined as a certain quotient of the  free product of a monoid $T$ and a semilattice $Y$, where $T$ acts on $Y$ from the left and from the right. 

Proper restriction semigroups generalize $E$-unitary inverse semigroups which have been widely studied in semigroup theory and have important applications far beyond it. The structure of the latter,  first established by McAlister in \cite{McA74I, McA74II}, can be equivalently described in terms of partial actions of  groups on semilattices \cite{KL04, PR79}. This construction can be smoothly extended to partial actions of monoids on semilattices
to describe the structure of proper restriction semigroups \cite{CG12, K15} (for further applications of partial actions of monoids on semilattices to the study of restriction semigroups see \cite{DKhK21,K19}). Similarly to the fact that the free inverse semigroup is $E$-unitary, the free restriction semigroup is proper. In addition, every restriction semigroup has a proper cover, which parallels the fact that every inverse semigroup has an $E$-unitary cover. In this paper we construct a class of Ehresmann semigroups based on a suitable generalization of the notion of a partial action of a monoid on a semilattice, such that the free Ehresmann semigroup belongs to this class and every Ehresmann semigroup has a cover that belongs to this class. 

We now describe the structure of the paper and highlight its main ideas and results. In Section \ref{s:prelim} we collect the basic definitions and facts. In Section \ref{s:structure} we define proper elements of an Ehresmann semigroup as elements which are uniquely determined by their domain projection, range projection and their $\sigma$-class (Definition \ref{def:proper_elem}). We then define proper Ehresmann semigroups as Ehresmann semigroups generated by an order ideal containing all the projections and consisting of proper elements with the property of the uniqueness of matching factorizations (Definition \ref{def:proper}). 
Recall that in proper restriction semigroups all elements are determined by their domain projection (or their range projection) and their $\sigma$-class \cite{CG12,K15}. However, in the free Ehresmann semigroup $FES(X)$ such two-coordinate description of elements fails (see, e.g., the remark after Question \ref{q:1}), whereas  our three-coordinate approach to generating elements works, and $FES(X)$ is proper (see Proposition~\ref{prop:free_proper}). Having defined proper Ehresmann semigroups, we proceed to a construction which, given a semilattice $E$, a monoid $T$ and a labelled directed graph ${\mathcal G}$ with vertex set $E$ and edges labelled by elements of $T$ possessing additional structure called compatible restrictions and corestrictions (Definition~\ref{def:labelled_graph}), outputs an Ehresmann semigroup, denoted $E\rtimes_{\mathcal G} T$, which, under an additional condition, is proper (Theorem \ref{th:construction}). Then in Theorem~\ref{th:main} we prove that every proper Ehresmann semigroup is isomorphic to one so constructed. In Section~\ref{s:covers} we prove that every Ehresmann semigroup has a proper cover (Theorem \ref{th:cover}). In Section \ref{s:connection} we show that in the case where $S$ is an Ehresmann monoid the covering monoid $P(S)\rtimes_{\mathcal G} X^*$ from the proof of Theorem \ref{th:cover} is isomorphic to the monoid ${\mathcal P}(X^*, P(S))$ from \cite{BGG15}. This observation leads to uniquely determined normal forms of elements of ${\mathcal P}(X^*, P(S))$,  which are different from the (not uniquely determined) normal forms proposed in \cite{BGG15}. As a consequence this implies that the free Ehresmann monoid and the free Ehresmann semigroup are proper (Proposition \ref{prop:free_proper}). In the final Section \ref{s:multiactions} we define partial multiactions of monoids on semilattices which are a special case of the previously defined labelled directed graphs with compatible restrictions and corestrictions and, at the same time, are a natural generalization of partial actions of monoids on semilattices. We first define partial multiactions of a monoid $T$ on a set $X$ (Definition~\ref{def:part_multiaction}) and show that they are in a bijection with premorphisms $T\to {\mathcal B}(X)$. We observe that a proper Ehresmann semigroup is strictly proper (that is, all its elements are proper) if and only if its underlying labelled directed graph reduces to a partial multiaction (Proposition \ref{prop:strictly}). In Corollary \ref{cor:special} we provide a structure result on proper left restriction Ehresmann semigroups. If $S$ is a restriction semigroup, its attached partial multiaction reduces to the usual partial action of a monoid on a semilattice by partial bijections between order ideals, and  we thus recover the known structure result on proper restriction semigroups \cite{CG12}, as it is formualted in \cite{K15}. We conclude the paper by two open questions about strictly proper Ehresmann semigroups.

\section{Preliminaries}\label{s:prelim}
\subsection{Ehresmann and restriction semigroups} 
We now define the main objects of study of this paper. For more information we refer the reader to the survey \cite{G10}, see also the recent papers on the subject \cite{BGG15,BGGW18,CDEGZ22}.

\begin{definition} [Ehresmann semigroups] A {\em left Ehresmann semigroup} is an algebra $(S; \cdot \,, ^+)$, where $(S;\cdot)$ is a semigroup and $^+$ is a unary operation satisfying the following identities:
\begin{equation}\label{eq:axioms_plus}
x^+x=x, \,\,\, (x^+y^+)^+=x^+y^+=y^+x^+, \,\,\, (xy)^+=(xy^+)^+.
\end{equation}

Dually, a {\em right Ehresmann semigroup} is an algebra $(S; \cdot \,, ^*)$, where $(S;\cdot)$ is a semigroup and $^*$ is a unary operation satisfying the following identities:
\begin{equation}\label{eq:axioms_star}
xx^*=x, \,\,\, (x^*y^*)^*=x^*y^*=y^*x^*, \,\,\, (xy)^*=(x^*y)^*.
\end{equation}

A {\em two-sided Ehresmann semigroup}, or just an {\em Ehresmann semigroup}, is an algebra $(S; \cdot\, , ^*, ^+)$, where $(S;\cdot\, , ^+)$ is a left Ehresmann semigroup, $(S;\cdot\, , ^*)$ is a right Ehresmann semigroup, and the operations $^*$ and $^+$ are connected by the following identities:
\begin{equation}\label{eq:axioms_common}
(x^+)^*=x^+,\,\,\, (x^*)^+=x^*.
\end{equation}
\end{definition}
		
Restriction semigroups form an important subclass of Ehresmann semigroups. They are defined as follows.	
\begin{definition}[Restriction semigroups] A {\em left restriction semigroup} is an algebra $(S; \cdot \,, ^+)$ which is a left Ehresmann semmigroup and in addition satisfies the identity
\begin{equation}\label{eq:axioms_plus_restr}
xy^+ = (xy)^+x.
\end{equation}

Dually, a {\em right restriction semigroup} is an algebra $(S; \cdot \,, ^*)$ which is a right Ehresmann semigroup and in addition satisfies the identity
\begin{equation}\label{eq:axioms_star_restr}
x^*y = y(xy)^*
\end{equation}
A {\em two-sided restriction semigroup}, or just a {\em restriction semigroup} is an algebra $(S; \cdot\, , ^*, ^+)$ which is a left and a right restriction semigroup and \eqref{eq:axioms_common} holds.
\end{definition}
		
Restriction semigroups, in turn, generalize inverse semigroups. Recall that a semigroup $S$ is called an {\em inverse semigroup} if for each $a\in S$ there exists unique
$b\in S$ such that $aba=a$ and $bab=b$. The element $b$ is called the {\em inverse} of $a$ and is denoted by $a^{-1}$. If $(S, \cdot)$ is an inverse semigroup and $a\in S$, we define $a^+=aa^{-1}$ and $a^*=a^{-1}a$. Then $(S; \cdot\, , ^*, ^+)$ is a restriction semigroup (and thus an Ehresmann semigroup).

An Ehresmann semigroup possessing an identity element is called an {\em Ehresmann monoid}. 
		
Left Ehresmann semigroups and right Ehresmann semigroups are usually considered as $(2,1)$-algebras and Ehresmann semigroups as $(2,1,1)$-algebras.  
It is immediate from the definition that left Ehresmann semigroups, right Ehresmann semigroups, Ehresmann semigroups, left restriction semigroups, right restriction semigroups and restriction semigroups form varieties of algebras. Morphisms, congruences and subalgebras of all these algebras are taken with respect to their signatures.
		
Let $S$ be an Ehresmann semigroup. In view of  \eqref{eq:axioms_plus}, \eqref{eq:axioms_star} and \eqref{eq:axioms_common}, the set 
$$
P(S)=\{s^*\colon s\in S\}=\{s^+\colon s\in S\}
$$
is closed with respect to the multiplication. Furthermore, it is a semilattice and  $e^*=e^+=e$ holds for all $e\in P(S)$.  It is called the {\em semilattice of projections} of $S$ and its elements are called {\em projections}. A projection is necessarily an idempotent, but an Ehresmann semigroup may contain idempotents that are not projections. 
		
We will often use the following identities which can be easily derived from the definitions: 
\begin{equation}\label{eq:rule1}
\forall s\in S, e\in P(S)\colon \, (se)^* = s^*e , \,\,\, (es)^+ = es^+.
\end{equation}

The next identities hold in restriction semigroups and are often called the {\em left ample identity} and the {\em right ample identity}, respectively:
\begin{equation}\label{eq:mov_proj}
\forall s\in S, e\in P(S)\colon \, se = (se)^+s, \,\,\, es = s(es)^*.
\end{equation}
		
Let $S$ be an Ehresmann semigroup. For $a,b\in S$, we put:
\begin{itemize}
\item $a\leq_l b$ if there is $e\in P(S)$ such that $a=eb$;
\item $a\leq_r b$ if there is $e\in P(S)$ such that $a=be$;
\item $a\leq b$ if there are $e,f \in P(S)$ such that $a=ebf$.
\end{itemize}
		
The relations $\leq_l$, $\leq_r$ and $\leq$ are partial orders on $S$ and are called the {\em natural left partial order}, the {\em natural right partial order} and the {\em natural partial order}, respectively. Clearly $\leq {} = {} \leq_l \circ \leq_r  = {} \leq_r \circ \leq_l$ where $\circ$ stands for the product of relations. It is easy to see that $a\leq_l b$ holds if and only if $a=a^+b$ and, dually, $a\leq_r b$ holds if and only if $a=ba^*$. Restricted to $P(S)$, all the orders coincide and $e\leq f$, where $e,f\in P(S)$, holds if and only if $e=ef$. In addition, $P(S)$ is an order ideal, that is, $e\leq f$ where $f\in P(S)$, implies that $e\in P(S)$. Moreover, if $S$ is a restriction semigroup, all the orders coincide, too.

The following is an easy but useful observation.
\begin{lemma}\label{lem:monoid} An Ehresmann semigroup $S$ is a monoid if and only if it has a maximum projection. If this is the case, the maximum projection is the identity element of~$S$.
\end{lemma}

\begin{proof} Suppose that $S$ is an Ehresmann monoid with the identity element $1$. Denote $e=1^*$. Then $e=1e = 11^* = 1$. Hence $1$ is a projection. Since $1e=e1=e$ for all $e\in P(S)$, it is the maximum projection. Conversely, suppose that an Ehresmann semigroup $S$ has a maximum projection $1_{P(S)}$ and let $s\in S$. Then $s1_{P(S)} = ss^*1_{P(S)} = ss^* =s$ and similarly $1_{P(S)}s=s$. It follows that $S$ is a monoid with the identity element $1_{P(S)}$.
\end{proof}

It is easy to see that in a left Ehresmann semigroup $a\leq_l b$ implies $a^+\leq b^+$. Indeed let $a=a^+b$. Then $a^+=(a^+b)^+=(a^+b^+)^+=a^+b^+$. 
		
A {\em reduced Ehresmann semigroup} is an Ehresmann semigroup $S$ for which $|P(S)|=1$. Then, necessarily, $S$ is a monoid and $P(S)=\{1\}$ so that  \mbox{$s^*=s^+=1$} holds for all $s\in S$. On the other hand, any monoid $S$ can be endowed with the structure of an Ehresmann semigroup by putting $s^*=s^+= 1$ for all $s\in S$. Hence reduced Ehresmann semigroups can be identified with monoids. 
		
By $\sigma$ we denote the minimum congruence on an Ehresmann semigroup $S$ that identifies all elements of $P(S)$. Observe that if $\tau$ is a semigroup congruence that identifies all of $P(S)$ then $a\mathrel{\tau} b$ of course implies $a^+ \mathrel{\tau} b^+$ and $a^* \mathrel{\tau} b^*$. So semigroup congruences and $(2,1,1)$-congruences identifying all the projections coincide. Thus $S/\sigma$ is the maximum quotient of $S$ which is a reduced Ehresmann semigroup. The $\sigma$-class which contains $a\in S$ will be denoted by $[a]_{\sigma}$. If $a\leq b$ then clearly $a\mathrel{\sigma} b$. If $S$ is a restriction semigroup, each of the following statements is equivalent to $s \, {\mathrel{\sigma}}\, t$  (see \cite[Lemma 8.1]{G}):
\begin{enumerate}
\item[(i)] there is $e\in P(S)$ such that $es = et$; 
\item[(ii)] there is $e \in P(S)$ such that $se = te$.
\end{enumerate}
		 
A left (respectively right) restriction semigroup is called {\em proper} if $a^+=b^+$ (respectively $a^*=b^*$) and $a \mathrel{\sigma} b$ imply that $a=b$.
A restriction semigroup is {\em proper} if it is proper as a left and as a right restriction semigroup. 

\subsection{The monoid of binary relations on a set and its submonoids}\label{subs:binary_rel} 
Let $X$ be a non-empty set. By ${\mathcal B}(X)$ we denote the {\em monoid of binary relations} on $X$ with the operation of the composition of relations.  Let $\tau\in {\mathcal B}(X)$ and $x,y\in X$. We write $x\tau = \{z\in X\colon (x,z)\in \tau\}$ and  $\tau y = \{z\in X\colon (z,y)\in\tau\}$. In case where $x\tau =\{y\}$ we write $x\tau =y$ and similarly if $\tau y=\{x\}$ we write $\tau y = x$. By ${\mathrm{id}}_X=\{(x,x)\colon x\in X\}$ we denote the {\em identity relation} and by $\varnothing$ we denote the {\em empty relation}. The {\em reverse relation} $\tau^{-1}$ of the relation $\tau\in {\mathcal B}(X)$ is defined by $\tau^{-1} = \{(y,x)\colon (x,y)\in \tau\}$.  The {\em partial transformation monoid} ${\mathcal{PT}}(X)$ is a submonoid of the monoid ${\mathcal B}(X)$ consisting of all $\tau\in {\mathcal B}(X)$ such that  $|x\tau|\leq 1$ for all $x\in X$. Let ${\mathcal{PT}}^c(X)$ be the submonoid of ${\mathcal B}(X)$ consisting of all $\tau \in {\mathcal B}(X)$ such that $|\tau y|\leq 1$ for all $y\in X$. Clearly, ${\mathcal{PT}}^c(X)$ is anti-isomorphic to ${\mathcal{PT}}(X)$ via the map $\tau\mapsto \tau^{-1}$. The {\em symmetric inverse monoid} ${\mathcal I}(X)$ equals ${\mathcal{PT}}(X)\cap {\mathcal{PT}}^c(X)$. It consists of all $\tau\in {\mathcal B}(X)$ such that that  $|x\tau|\leq 1$  and $|\tau x|\leq 1$ for all $x\in X$. 

It is well known that ${\mathcal B}(X)$ is an Ehresmann monoid if one defines $\tau^+$ and $\tau^*$ by
$$\tau^+=\dom(\tau)=\{(x,x)\in X\times X\colon \exists y\in X \text{ such that }(x,y)\in \tau\},$$
$$\tau^*=\ran(\tau)=\{(y,y)\in X\times X\colon \exists x\in X \text{ such that }(x,y)\in \tau\}.$$ 
Note that $P({\mathcal B}(X)) = \{\tau\in {\mathcal B}(X)\colon \tau \subseteq id_X\}$, hence the semilattice $P({\mathcal B}(X))$ is isomorphic to the powerset of $X$ with respect to the operation of intersection of subsets. It follows that ${\mathcal{PT}}(X)$, ${\mathcal{PT}}^c(X)$ and ${\mathcal I}(X)$ are $(\cdot, ^+,^*,1)$-subalgebras of ${\mathcal B}(X)$. It is well known and easy to check that ${\mathcal{PT}}(X)$ is a left restriction monoid. Dually, ${\mathcal{PT}}^c(X)$ is a right restriction monoid. Furthermore, ${\mathcal{I}}(X)$ is a restriction monoid (of course, ${\mathcal{I}}(X)$ is also an inverse monoid).

In this paper we consider ${\mathcal B}(X)$, ${\mathcal{PT}}(X)$, ${\mathcal{PT}}^c(X)$ and ${\mathcal I}(X)$ equipped with the partial order $\subseteq$ of inclusion of 
relations. It is known and easy to check that this partial order is compatible with the multiplication. The inclusion order on each of ${\mathcal{PT}}(X)$, ${\mathcal{PT}}^c(X)$ and ${\mathcal I}(X)$ coincides with the natural partial order. However, the inclusion order on ${\mathcal B}(X)$ contains the natural partial order $\leq$ but does not coincide with it. For example, if $X=\{1,2,3,4\}$, $\tau = \{(1,3), (1,4),(2,3),(2,4)\}$ and $\mu = \{(1,3), (2,4)\}$ then $\mu\subseteq \tau$ but $\mu\not\leq \tau$. 

\section{The structure of proper Ehresmann semigroups}\label{s:structure}
The goal of this section is to define proper Ehresmann semigroups and show that they arise naturally from labelled directed graphs with vertices indexed by elements of a semilattice  and edges labelled by elements of a monoid, with an additional structure called compatible restrictions and corestrictions. 

\subsection{Matching factorizations} Let $S$ be an Ehresmann semigroup. 

A {\em matching factorization} of an element $s\in S$ is a tuple $(s_1,\dots, s_n)$, where $n\in {\mathbb N}$ and $s_1,\dots, s_n\in S$ are such that 
$s=s_1\cdots s_n$ and $s_i^* = s_{i+1}^+$ for all $i=1, \dots, n-1$. 
 We say that a matching factorization $(s_1,\dots, s_n)$ of $s\in S$ has $n$ factors.

We record several easy properties of matching factorizations.

\begin{lemma}\label{lem:matching1}
Let  $(s_1,\dots, s_n)$ be a matching factorization of $s\in S$. Then $s^+ = s_1^+$ and $s^* = s_n^*$.
\end{lemma}

\begin{proof} We apply induction on $n$. If $n=1$, there is nothing to prove. Suppose that $n\geq 2$ and that the statement holds for factorizations  into $n-1$ factors. We then have
$s^+ = (s_1s_2\cdots s_n)^+  = (s_1(s_2\cdots s_n)^+)^+ = (s_1s_2^+)^+ = (s_1s_1^*)^+ = s_1^+$ and similarly $s^* = s_n^*$. 
\end{proof}

\begin{lemma}\label{lem:matching2} Let $(s_1,\dots, s_n)$ be a matching factorization of $s\in S$ and let $e\leq s^*=s_n^*$. Put $s_n' = s_ne$, $s_{n-1}' = s_{n-1}(s_n')^+$, $\dots,$ $s_1' = s_1(s_2')^+$. Then $(se)^* = e$ and $(s_1',\dots, s_n')$ is a matching factorization of $se$.
\end{lemma}

\begin{proof} Note that $e\leq s^*$ implies that $e\in P(S)$. We apply induction on $n$. If $n=1$, there is nothing to prove. Suppose that $n\geq 2$ and that the statement holds for factorizations  into $n-1$ factors.  Note that, in view of \eqref{eq:rule1} and \eqref{eq:axioms_plus}, we have $s_n^+(s_ne)^+ = (s_n^+s_ne)^+ = (s_ne)^+$. It follows that $(s_n')^+\ = (s_ne)^+ \leq s_n^+ = s_{n-1}^*$. By the induction hypothesis we have that $(s_1',\dots, s_{n-1}')$ is a matching factorization of $t=s_1\cdots s_{n-1}(s_n')^+$. 
Since, using \eqref{eq:rule1}, $(s_{n-1}')^* = (s_{n-1}(s_n')^+)^* = (s_{n-1}')^* (s_n')^+ = (s_n')^+$, it follows that $(s_1',\dots, s_n')$ is a matching factorization of $se$. In addition, $(se)^*  = (s_ne)^* = s_n^*e = e$.
\end{proof}

\begin{lemma}\label{lem:matching3}
Let $s=s_1\cdots s_n$. Then there are $s_1'\leq s_1$, $\dots$, $s_n'\leq s_n$ such that $(s_1',\dots, s_n')$ is a matching factorization of $s$.
\end{lemma}

\begin{proof} We apply induction on $n$. If $n=1$ then $s=s_1$ is trivially a matching factorization. Suppose that $n\geq 2$ and that the statement holds for factorizations  into $n-1$ factors. Let  $t=s_1\cdots s_{n-1}$, $s_n' = t^*s_n$ and $e= t^*s_n^+$. Then $s= ts_n = (ts_n^+) s_n'=tes_n'$ where
$(ts_n^+)^* = (s_n')^+ = e$. By the induction hypothesis there are $t_1\leq s_1$, $\dots,$ $t_{n-1}\leq s_{n-1}$ such that $(t_1,\dots, t_{n-1})$ is a matching factorization of $t$.
Because $t_{n-1}^* = t^*$ by Lemma \ref{lem:matching1}, we have $e=t^*s_n^+ = t_{n-1}^*s_n^+\leq t_{n-1}^*$.
By Lemma \ref{lem:matching2} there are $s_1'\leq t_1$, $\dots,$ $s_{n-1}'\leq t_{n-1}$ such that $(s_1',\dots, s_{n-1}')$ is a matching factorization of $te$. Since $e\leq t^*$, it follows from Lemma \ref{lem:matching1} that $(s_{n-1}')^* = (te)^* = t^*e= e = (s_n')^+$. Hence $(s_1',\dots, s_{n}')$ is a matching factorization of $s$ with $s_1'\leq s_1$, $\dots,$ $s_n'\leq s_n$, as required.
\end{proof}

\subsection{Proper Ehresmann semigroups} Let $S$ be an Ehresmann semigroup. Recall that by $\sigma$ we denote the minimum congruence on $S$ that identifies all the elements of $P(S)$.
\begin{definition}[Proper elements]\label{def:proper_elem} An element $s\in S$ will be called {\em proper}, if 
$$
\forall t\in S\colon t^+=s^+, \, t^*=s^* \text{ and } t\mathrel{\sigma} s \text{ imply that } t=s.
$$
\end{definition}

In other words, an element  $s$ is  proper  if it is uniquely determined by $s^+$, $s^*$ and $[s]_{\sigma}$.

If $(s_1, \dots, s_n)$ is a matching factorization where all $s_i$ are proper (or all $s_i$ belong to some subset $Y$ of $S$, etc.), we will say that $(s_1, \dots, s_n)$ is a matching factorization of $s$ {\em into a product of proper elements} (resp. {\em into a product of elements of $Y$}, etc.).

Let $s\in S$ and $(s_1, \dots, s_n)$ be a matching factorization of $s$ into a product of  proper elements. If there are $i,j\in \{1,2,\dots, n\}$ where $i\leq j$ such that $s_i\cdots s_j = t$ is a proper element then $(s_1,\dots, s_{i-1}, t, s_{j+1} ,\dots, s_n)$ is also a matching factorization of $s$ into a product of proper elements. In this case we write $(s_1,\dots, s_n) \to (s_1,\dots, s_{i-1}, t, s_{j+1}, \dots, s_n)$. Clearly the relation $\to$ is reflexive. Let $\equiv$ be its symmetric and transitive closure. 

We say that two matching factorizations $(s_1,\dots, s_n)$ and $(t_1,\dots, t_k)$ of $s$ into products of proper elements are {\em equivalent} if $(s_1,\dots, s_k) \equiv (t_1,\dots, t_k)$, that is, if the factorization  $(t_1,\dots, t_k)$ can be obtained from the factorization $(s_1,\dots, s_k)$ by a finite number of applications of the relation $\to$ or its inverse relation.

\begin{definition}[Proper Ehresmann semigroups] \label{def:proper} An Ehresmann semigroup $S$ is called {\em proper} if there is an order ideal $Y$ of $S$, with respect to the natural partial order, called a {\em proper generating ideal} of $S$, such that
\begin{enumerate}
\item $P(S)\subseteq Y$;
\item  all elements of $Y$ are proper; 
\item each $s\in S$ admits a unique, up to equivalence, matching factorization as a product of elements of $Y$.
\end{enumerate}
\end{definition}

We say that $S$ is {\em strictly proper} if all of its elements are proper. It is easy to see that a strictly proper Ehresmann semigroup is proper (in Definition \ref{def:proper} one just takes $Y=S$). 

It is immediate that a proper restriction semigroup $S$ is strictly proper as an Ehresmann semigroup. It follows that proper and strictly proper Ehresmann semigroups generalize proper restriction semigroups.

\subsection{Labelled directed graphs with compatible restrictions and corestrictions and their attached categories of paths}\label{subs:categories}

\begin{definition}[Labelled directed graphs] \label{def:labelled_graph_sets}  
Let $X$ be a non-empty set, $T$ a monoid with the identity element $1$ and ${\mathcal G}$ a labelled directed graph with vertex set $X$ and edges labelled by elements of $T$. The edge set of ${\mathcal G}$ is denoted by ${\mathrm E}({\mathcal G})$. For $c\in {\mathrm E}({\mathcal G})$ by ${\mathbf{d}}(c),{\mathbf{r}}(c)\in X$ and $t={\mathbf{l}}(c)\in T$ we denote its initial vertex, terminal vertex and label, respectively. We assume that for every $x,y\in X$ and $t\in T$ there is at most one edge in ${\mathcal G}$ with initial vertex $x$, final vertex $y$ and label $t$, that is, every edge $c\in  {\mathrm E}({\mathcal G})$ is uniquely determined by  ${\mathbf{d}}(c),{\mathbf{r}}(c)$ and ${\mathbf{l}}(c)$. If ${\mathcal G}$ has an edge $c$ with initial vertex $x$, terminal vertex $y$ and label $t$, we denote this edge by $(x,t,y)$. We require that for each $x\in X$ there is an edge $(x,1,x)\in {\mathrm E}({\mathcal G})$. 
\end{definition} 

By the {\em label} of a path in ${\mathcal G}$ we mean the product of labels of the consequtive edges of this path.

\begin{definition}\label{def:labelled_graph} (Labelled directed graphs with compatible restrictions and corestrictions) Let now $E$ be a semilattice, $T$ a monoid with the identity element $1$ and ${\mathcal G}$ a labelled directed graph with vertex set $E$ and edges labelled by elements of $T$ (see Definition \ref{def:labelled_graph_sets}). We assume that for every $t\in T$ there is a path in ${\mathcal G}$ labelled by $t$. We say that ${\mathcal G}$ has {\em restrictions} if for any $(e,t,f)\in {\mathrm E}({\mathcal G})$ and $g\leq e$ there exists $_{g}|(e,t,f)\in  {\mathrm E}({\mathcal G})$, called the {\em restriction of} $(e,t,f)$ {\em to} $g$ such that  the following axioms hold:
\begin{itemize}
\item[(R1)] \label{i:restriction1}  If $(e,t,f)\in {\mathrm E}({\mathcal G})$ and $g\leq e$ then $_{g}|(e,t,f) = (g, t, f')$ for some $f'\leq f$.
\item[(R2)] \label{i:restriction2} If $(e,t,f)\in {\mathrm E}({\mathcal G})$ then $_{e}|(e,t,f) = (e,t,f)$.
\item[(R3)] \label{i:restriction3} If  $(e,t,f)\in {\mathrm E}({\mathcal G})$ and $h\leq g\leq e$ then $_h|(_g|(e,t,f)) = {}_h|(e,t,f)$.
\item[(R4)] \label{i:restriction4} Let $n\geq 2$ and $p_1=(e_0,t_1,e_1), \dots, p_n=(e_{n-1},t_n,e_n) \in {\mathrm E}({\mathcal G})$ be such that $(e_0,t_1\cdots t_n,e_n) \in  {\mathrm E}({\mathcal G})$. Let also $e_0'\leq e_0$. Denote $e_{1}' = {\mathbf r}(_{e_0'}|p_1),$ $e_{2}' = {\mathbf r}(_{e_1'}|p_{2}),\dots,$ $e_n' = {\mathbf r}(_{e_{n-1}'}|p_n).$
Then $_{e_0'}|(e_0,t_1\cdots t_n,e_n)=(e_0', t_1\cdots t_n, e_n')$.
\item[(R5)]\label{i:restriction5} For any $e,f\in E$ such that $f\leq e$ we have $_f|(e,1,e) = (f,1,f)$.
\end{itemize}
    
Dually, we say that ${\mathcal G}$ has {\em corestrictions} if for any $(e,t,f)\in {\mathrm E}({\mathcal G})$  and $g\leq f$ there exists $(e,t,f)|_g\in {\mathrm E}({\mathcal G})$, called the {\em corestriction of} $(e,t,f)$ {\em to} $g$ such that following axioms hold:
\begin{itemize}
\item[(CR1)] \label{i:corestriction1}  If $(e,t,f)\in {\mathrm E}({\mathcal G})$ and $g\leq f$ then $(e,t,f)|_g = (e', t, g)$ for some $e'\leq e$.
\item[(CR2)] \label{i:corestriction2} If $(e,t,f)\in {\mathrm E}({\mathcal G})$ then $(e,t,f)|_f = (e,t,f)$.
\item[(CR3)] \label{i:corestriction3} If  $(e,t,f)\in {\mathrm E}({\mathcal G})$ and $h\leq g\leq f$ then $((e,t,f)|_g)|_h = (e,t,f)|_h$.
\item[(CR4)] \label{i:corestriction4} Let $n\geq 2$ and $p_1=(e_0,t_1,e_1), \dots, p_n=(e_{n-1},t_n,e_n) \in {\mathrm E}({\mathcal G})$ be such that $(e_0,t_1\cdots t_n,e_n) \in  {\mathrm E}({\mathcal G})$. Let also $e_n'\leq e_n$. Denote $e_{n-1}' = {\mathbf d}(p_n|_{e_n'}),$ $e_{n-2}' = {\mathbf d}(p_{n-1}|_{e_{n-1}'}),\dots,$ $e_0' = {\mathbf d}(p_1|_{e_1'}).$
Then $(e_0,t_1\cdots t_n,e_n)|_{e_n'}=(e_0', t_1\cdots t_n, e_n')$.
\item[(CR5)] For any $e,f\in E$ such that $f\leq e$ we have $(e,1,e)|_f = (f,1,f)$.
\end{itemize}
    
Suppose that ${\mathcal G}$ has restrictions and corestrictions. We say that it has {\em compatible restrictions and corestrictions} if the following axiom holds:
   
\begin{itemize}
\item[(C)] Let $c=(e,t,f)\in {\mathrm E}({\mathcal G})$ and $g\leq e$, $h\leq f$. Then 
$$(_{g}|c)|_{{\mathbf r}(_{g}|c)h} = \space{} _{{\mathbf d}(c|_h)g}|(c|_h) 
= \left(g {\mathbf d}(c|_h), t, {\mathbf r}(_{g}|c)h\right).
$$
\end{itemize}
\end{definition}

The following diagrams illustrate the operations of restriction and corestriction. 
\begin{center}
\begin{tikzpicture}
\tikzset{edge/.style = {->,> = latex'}}
\node (A) at (0,1.6) {$e$};
\node (B) at (2,1.6) {$f$};
\node (C) at (0,0) {$g$};
\node (D) at (2,0) {$f'$};
\node (E) at (6,1.6) {$e$};
\node (F) at (8,1.6) {$f$};
\node (G) at (6,0) {$e'$};
\node (H) at (8,0) {$h$};
\node at (1,2) {$(e,t,f)$}; 
\node at (1,-0.4) {${_g|(e,t,f)}$};
\node at (7,2) {$(e,t,f)$}; 
\node at (7,-0.4) {$(e,t,f)|_h$}; 
\node at (0,0.8) {\vertleq};
\node at (2,0.8) {\vertleq};
\node at (6,0.8) {\vertleq};
\node at (8,0.8) {\vertleq};
\draw [edge, thick]  (A) -- (B) ;
\draw [edge, thick]  (E) -- (F) ;
\draw [edge, thick]  (G) -- (H) ;
\draw [edge, thick]  (C) -- (D) ;
\end{tikzpicture}
\end{center}

From now till the end of this section let $E$ be a semilattice, $T$ a monoid and ${\mathcal G}$ a labelled directed graph with vertex set $E$ and edges labelled by elements $T$ which has compatible restrictions and corestrictions. By ${{\mathcal C}}({\mathcal G})$ we denote the {\em path semicategory} of ${\mathcal G}$. By definition, the objects of ${{\mathcal C}}({\mathcal G})$ are elements of $E$ and its arrows are non-empty (labelled) directed paths in ${\mathcal G}$.  If $p$ is a path in ${\mathcal G}$ with consecutive edges $(e_0, t_1, e_1), \dots, (e_{n-1}, t_n, e_n)$, $n\geq 1$, we sometimes denote it by $p=(e_0, t_1, e_1, \dots, e_{n-1}, t_n, e_n)$. We say that $n$ is the {\em length} of the path $p$. We  put $e_0={\mathbf d}(p)$ and $e_n={\mathbf r}(p)$ to be the initial and the terminal vertices of $p$, respectively, and ${\mathbf l}(p) = t_1\cdots t_n$ to be the label of $p$. Paths $p$ and $q$\ are {\em composable} if ${\mathbf r}(p) = {\mathbf d}(q)$, in which case their product is their concatenation, denoted $pq$.  

We extend the definition of restriction and corestriction from edges of ${\mathcal G}$ to non-empty directed paths in ${\mathcal G}$ recursively as follows: let $p=(e_0, t_1, e_1, \dots, e_{n-1}, t_n, e_n)$ be a path and put $p_1=(e_0, t_1, e_1)$, $\dots,$ $p_n=(e_{n-1}, t_n, e_n)$. 

\begin{enumerate}
\item[(RPath)] Let $e_0'\leq e_0$. We put $e_{1}' = {\mathbf r}(_{e_0'}|p_1),$ $e_{2}' = {\mathbf r}(_{e_{1}'}|p_{2}),$ $\dots,$ $e_n' = {\mathbf r}(_{e_{n-1}'}|p_n).$
We then define $_{e_0'}|p = (e_0',t_1,e_1',t_2,e_2',\dots, e_{n-1}', t_n, e_n')$.
\item[(CRPath)] Let $e_n'\leq e_n$. We put $e_{n-1}' = {\mathbf d}(p_n|_{e_n'}),$ $e_{n-2}' = {\mathbf d}(p_{n-1}|_{e_{n-1}'}),$ $\dots,$ $e_0' = {\mathbf d}(p_1|_{e_1'}).$ We then define $p|_{e_n'}=
(e_0',t_1,e_1',t_2,e_2',\dots, e_{n-1}', t_n, e_n')$.
\end{enumerate}

In the following lemma we extend axioms (R1)--(R5), (CR1)--(CR5) and (C) from edges of ${\mathcal G}$ to non-empty paths in ${\mathcal G}$.

\begin{lemma}\label{lem:axioms_paths}
Let $p\in {{\mathcal C}}({\mathcal G})$ and $e\leq {\mathbf d}(p), f\leq {\mathbf r}(p)$. Then:
\begin{enumerate}
\item[${\mathrm{(R1a)}}$]\label{i:paths1} ${\mathbf r}(_e|p)\leq {\mathbf r}(p)$;
\item[${\mathrm{(R2a)}}$]\label{i:paths2} $_{{\mathbf d}(p)}|p = p$;
\item[${\mathrm{(R3a)}}$]\label{i:paths3} if $g\leq e$ then $_g|(_e|p) = {} _g|p$;
\item[${\mathrm{(R4a)}}$]\label{i:paths4} if $q\in {{\mathcal C}}({\mathcal G})$ is such that ${\mathbf r}(p) = {\mathbf d}(q)$ then $_e|(pq) = {}_e|p \,\, {}_{{\mathbf r}(_e|p)}|q$;
\item[${\mathrm{(R5a)}}$]\label{i:paths5} if $p=(e,1,e)^n$ where $n\geq 1$ and $g\leq e$ then $_g|p = (g,1,g)^n$;
\item[${\mathrm{(CR1a)}}$]\label{i:paths1c} ${\mathbf d}(p|_f)\leq {\mathbf d}(p)$;
\item[${\mathrm{(CR2a)}}$]\label{i:paths2c} $p|_{{\mathbf r}(p)} = p$;
\item[${\mathrm{(CR3a)}}$]\label{i:paths3c} if $g\leq f$ then $(p|_f)|_g = {} p|_g$;
\item[${\mathrm{(CR4a)}}$]\label{i:paths4c}  if $q\in {{\mathcal C}}({\mathcal G})$ is such that ${\mathbf r}(p) = {\mathbf d}(q)$ and $g\leq {\mathbf r}(q)$ then $(pq)|_g = p|_{{\mathbf d}(q|_g)} \,\, q|_g$;
\item[${\mathrm{(CR5a)}}$]\label{i:paths5c} if $p=(e,1,e)^n$ where $n\geq 1$ and $g\leq e$ then $p|_g = (g,1,g)^n$;
\item[${\mathrm{(Ca)}}$] \label{i:paths_comp} $(_e|p)|_{{\mathbf r}(_e|p)f} = \space{} _{{\mathbf d}(p|_f)e}|(p|_f)$.
\end{enumerate} 
\end{lemma}

Remark that (R4a) and (CR4a) can be extended to products of more than two paths.

\begin{proof}  (R4a) Let $p= p_1\cdots p_k$, $q=p_{k+1}\cdots p_n$ where $p_i = (e_{i-1},t_{i}, e_i)$ for all $i=1,\dots, n$. Put $e_0'=e$, $e_1' = {\mathbf r}(_{e_0'}|p_1)$, $\dots$, $e_n' = {\mathbf r}(_{e_{n-1}'}|p_{n})$. By (RPath) we have that $_e|(pq) = (e_0', t_1, e_1',\dots, e_{n-1}', t_n, e_n')$, and also $$
_e|p = (e_0', t_1, e_1',\dots, e_{k-1}', t_k, e_k'), \,\,\, {}_{{\mathbf r}(_e|p)}|q = 
{}_{e_k'}|q = (e_k', t_{k+1}, e_{k+1}', \dots, e_{n-1}', t_n, e_n').$$ 
The equality $_e|(pq) = {}_e|p \,\, {}_{{\mathbf r}(_e|p)}|q$ follows.

All other items (excluding (Ca)) can be proved easier or similarly, and we leave the details to the reader.

We finally prove (Ca). If $p$ has length $1$, then the statement holds by (C). 
Suppose that $n\geq 2$ and that (Ca) holds for all paths of length $n-1$. 
Denote $p=rp_n$ where $r=(e_0,t_1,e_1,\dots, e_{n-2}, t_{n-1}, e_{n-1})$ and $p_n = (e_{n-1}, t_n, e_n)$. Put $e_0'=e$, $e_{n-1}' = {\mathbf r}(_{e_0'}|r)$ and $e_n' =  {\mathbf r}(_{e_{n-1}'}|p_n)$. Put also $e_n'' = f$ and $e_{n-1}'' = {\mathbf d}(p_n|_{e_n''})$, $e_0'' = {\mathbf d}(r|_{e_{n-1}''})$. Observe that ${\mathbf r}(_{e_0'}|p) =  e_n'$. Indeed,  $_{e_0'}|p = {} _{e_0'}|(rp_n)$ which by (R4a) equals $_{e_0'}|r\, _{{\mathbf r}(_{e_0'}|r)}|p_n = {}_{e_0'}|r\, _{e_{n-1}'}|p_n$. It follows that ${\mathbf r}(_{e_0'}|p) = {\mathbf r}(_{e_{n-1}'}|p_n) = e_n'$. Similarly, we have ${\mathbf d}(p|_{e_n''}) =  e_0''$. 
This can be illustrated with the following diagram:
\begin{center}
\begin{tikzpicture}
\tikzset{edge/.style = {->,> = latex'}}
\node (E0) at (0,1.6) {$e_0$};
\node (E1) at (2,1.6) {$e_1$};
\node (E2) at (4,1.6) {$e_2$};
\node (Edots) at (6,1.6) {$\cdots$};
\node (En-1) at (8,1.6) {$e_{n-1}$};
\node (En) at (10,1.6) {$e_n$};
\draw [edge, thick]  (E0) -- (E1) ;
\draw [edge, thick]  (E1) -- (E2) ;
\draw [edge, thick]  (En-1) -- (En) ;
\node at (1,2) {$t_1$}; 
\node at (3,2) {$t_2$};
\node at (9,2) {$t_n$}; 
\node (A0) at (0,3.2) {$e=e'_0$};
\node (A1) at (2,3.2) {$e'_1$};
\node (A2) at (4,3.2) {$e'_2$};
\node (Adots) at (6,3.2) {$\cdots$};
\node (An-1) at (8,3.2) {$e'_{n-1}$};
\node (An) at (10,3.2) {$e'_n$};
\draw [edge, thick]  (A0) -- (A1) ;
\draw [edge, thick]  (A1) -- (A2) ;
\draw [edge, thick]  (An-1) -- (An) ;
\node at (1,3.6) {$t_1$}; 
\node at (3,3.6) {$t_2$};
\node at (9,3.6) {$t_n$}; 
\node (B0) at (0,0) {$e''_0$};
\node (B1) at (2,0) {$e''_1$};
\node (B2) at (4,0) {$e''_2$};
\node (Bdots) at (6,0) {$\cdots$};
\node (Bn-1) at (8,0) {$e''_{n-1}$};
\node (Bn) at (10,0) {$e''_n=f$};
\draw [edge, thick]  (B0) -- (B1) ;
\draw [edge, thick]  (B1) -- (B2) ;
\draw [edge, thick]  (Bn-1) -- (Bn) ;
\node at (1,0.4) {$t_1$}; 
\node at (3,0.4) {$t_2$};
\node at (9,0.4) {$t_n$}; 
\node at (0,0.8) {\vertleq};
\node at (2,0.8) {\vertleq};
\node at (4,0.8) {\vertleq};
\node at (8,0.8) {\vertleq};
\node at (10,0.8) {\vertleq};
\node at (0,2.4) {\vertgeq};
\node at (2,2.4) {\vertgeq};
\node at (4,2.4) {\vertgeq};
\node at (8,2.4) {\vertgeq};
\node at (10,2.4) {\vertgeq};
\end{tikzpicture}
\end{center}
Applying (C) and the inductive hypothesis we have 
\begin{equation}\label{eq:j1}
(_{e_{n-1}'}|p_n)|_{e_n'e_n''} = {}_{e_{n-1}'e_{n-1}''}|(p_n|_{e_n''}),
\end{equation} 
\begin{equation}\label{eq:j2}
 {}_{e_0'e_0''}|(r|_{e_{n-1}''}) = (_{e_0'}|r)|_{e_{n-1}'e_{n-1}''}.
\end{equation}
It follows that
\begin{align*}
(_{e_0'}|p)|_{e_n'e_n''} & = ((_{e_0'}|r)(_{e_{n-1}'}|p_n))|_{e_n'e_n''}   & (\text{by (R4a)}) \\
& = (_{e_0'}|r)|_{e_{n-1}'e_{n-1}''} \,\, {}_{e_{n-1}'e_{n-1}''}|(p_n|_{e_n''}) & (\text{by (CR4a) and } \eqref{eq:j1})\\
& = {}_{e_0'e_0''}|(r|_{e_{n-1}''}) \,  {}_{e_{n-1}'e_{n-1}''}|(p_n|_{e_n''}) & (\text{by \eqref{eq:j2}})\\
& = {}_{e_0'e_0''}|(r|_{e_{n-1}''}p_n|_{e_n''}) & (\text{by (R4a) and \eqref{eq:j2}})\\
& = {}_{e_0'e_0''}|(p|_{e_n''}),
\end{align*}
which completes the proof.
\end{proof}

For any edges $p_1=(e_0,t_1,e_1), \dots, p_n=(e_{n-1},t_n,e_n) \in {\mathrm E}({\mathcal G})$ such that $(e_0,t_1\cdots t_n,e_n) \in  {\mathrm E}({\mathcal G})$ where $n\geq 2$ we put $p_1\cdots p_n \equiv (e_0,t_1\cdots t_n,e_n)$. Note that  if $p\equiv q$ then ${\mathbf d}(p) = {\mathbf d}(q)$, ${\mathbf r}(p) = {\mathbf r}(q)$. 
Let $\sim$ be the congruence on  the semicategory 
${{\mathcal C}}({\mathcal G})$ generated by the relation $\equiv$. By ${{\widetilde{\mathcal C}}}({\mathcal G})$ we denote the quotient semicategory of ${{\mathcal C}}({\mathcal G})$ by $\sim$. For any $c\in {{\widetilde{\mathcal C}}}({\mathcal G})$ we have that  ${\mathbf d}(c) = {\mathbf d}(p)$ and ${\mathbf r}(c) = {\mathbf r}(p)$ where $p$ is an arbitrary representative of $c$.  In addition, since $p\sim q$ implies ${\mathbf l}(p) = {\mathbf l}(q)$, for any $c\in {{\widetilde{\mathcal C}}}({\mathcal G})$ we can define ${\mathbf l}(c)$ to be equal to ${\mathbf l}(p)$ where $p$ is an arbitrary representative of $c$. For $p\in {{\mathcal C}}({\mathcal G})$ by $[p]$ we denote the $\sim$-class that contains $p$.
Observe that ${{\widetilde{\mathcal C}}}({\mathcal G})$ is in fact a category. Indeed, if $e$ is a vertex of ${\mathcal G}$ then $[(e,1,e)]\in {{\widetilde{\mathcal C}}}({\mathcal G})$ is the identity morphism at $e$ in ${{\widetilde{\mathcal C}}}({\mathcal G})$.

\begin{lemma} Let $p,q\in {{\mathcal C}}({\mathcal G})$ be such that $p\sim q$ and let $e\leq {\mathbf d}(p) = {\mathbf d}(q)$ and $f\leq {\mathbf r}(p) = {\mathbf r}(q)$. Then $_e|p \sim {} _e|q$ and $p|_f\sim q|_f$. 
\end{lemma}

\begin{proof} It is enough to assume that $p=ap_1\cdots p_nb$  where $n\geq 2$, $p_i=(e_{i-1}, t_i, e_i)$, $i\in \{1,\dots, n\}$, and $q= a(e_0, t_1\cdots t_n, e_n)b$, $a,b\in  {{\mathcal C}}({\mathcal G})$.  Let $h= {\mathbf r}(_e|a)$. Then $$_e|p = {}_e|a \, {}_h|(p_1\cdots p_nb) \, \text{ and } \, _e|q = {}_e|a \, {}_h|((e_0, t_1\cdots t_n, e_n)b)$$ by (R4a). Hence it suffices to show that 
$_h|(p_1\cdots p_nb) \sim {}_h|((e_0, t_1\cdots t_n, e_n)b)$.  By the definition of $\sim$ we have $p_1\cdots p_n \sim (e_0, t_1\cdots t_n, e_n)$ thus, applying (R4) and (RPath), also $_h|(p_1\cdots p_n)\sim {}_h|(e_0, t_1\cdots t_n, e_n)$. Let $h'={\mathbf r}(_h|(p_1\dots p_n)) = {\mathbf r}(_h|(e_0, t_1\cdots t_n, e_n))$. Then $$_h|(p_1\cdots p_nb) = {}_h|(p_1\dots p_n)\, _{h'}|b \sim {}_h|(e_0, t_1\cdots t_n, e_n)\,_{h'}|b =  {}_h|((e_0, t_1\cdots t_n, e_n)b),$$ as needed. 

For $p|_f\sim q|_f$, the proof is similar.
\end{proof}

It follows that for $c\in {{\widetilde{\mathcal C}}}({\mathcal G})$ and $e\leq {\mathbf d}(c), f\leq {\mathbf r}(c)$ we can define the {\em restriction of} $c$ {\em to} $e$ by $_e|c = [_e|p]$ and the {\em corestriction of} $c$ to $f$ by $c|_f=[p|_f]$ where $p$ is an arbitrary representative of $c$.

The following is an analogue of Lemma \ref{lem:axioms_paths} for ${{\widetilde{\mathcal C}}}({\mathcal G})$.

\begin{lemma} \label{lem:axioms_path_cat}
Let $c\in {{\widetilde{\mathcal C}}}({\mathcal G})$ and $e\leq {\mathbf d}(c), f\leq {\mathbf r}(c)$. Then:
\begin{enumerate}
\item[${\mathrm{(R1b)}}$]\label{i:paths_cat1} ${\mathbf r}(_e|c)\leq {\mathbf r}(c)$;
\item[${\mathrm{(R2b)}}$]\label{i:paths_cat2} $_{{\mathbf d}(c)}|c = c$;
\item[${\mathrm{(R3b)}}$]\label{i:paths_cat3} if $g\leq e$ then $_g|(_e|c) = {} _g|c$;
\item[${\mathrm{(R4b)}}$]\label{i:paths_cat4} if $d\in {{\widetilde{\mathcal C}}}({\mathcal G})$ is such that ${\mathbf r}(c) = {\mathbf d}(d)$ then $_e|(cd) = {}_e|c \, {}_{{\mathbf r}(_e|c)}|d$;
\item[${\mathrm{(R5b)}}$]\label{i:paths_cat5} if $c=[(e,1,e)]$ and $g\leq e$ then $_g|c = [(g,1,g)]$;
\item[${\mathrm{(CR1b)}}$]\label{i:paths_cat1c} ${\mathbf d}(c|_f)\leq {\mathbf d}(c)$;
\item[${\mathrm{(CR2b)}}$]\label{i:paths_cat2c} $c|_{{\mathbf r}(c)} = c$;
\item[${\mathrm{(CR3b)}}$]\label{i:paths_cat3c} if $g\leq f$ then $(c|_f)|_g = {} c|_g$;
\item[${\mathrm{(CR4b)}}$]\label{i:paths_cat4c}  if $d\in {{\widetilde{\mathcal C}}}({\mathcal G})$ is such that ${\mathbf r}(c) = {\mathbf d}(d)$ and $g\leq {\mathbf r}(d)$ then $(cd)|_g = c|_{{\mathbf d}(d|_g)} \, d|_g$;
\item[${\mathrm{(CR5b)}}$]\label{i:paths_cat5c} if $c=[(e,1,e)]$ and $g\leq e$ then $c|_g = [(g,1,g)]$;
\item[${\mathrm{(Cb)}}$]   $(_e|c)|_{{\mathbf r}(_e|c)f} = \space{} _{{\mathbf d}(c|_f)e}|(c|_f)$.
\end{enumerate} 
\end{lemma}

\begin{proof} The proofs follow from (R1a)--(R5a), (CR1a)--(CR5a) and (Ca) using the definition of restriction and corestriction in ${{\widetilde{\mathcal C}}}({\mathcal G})$.
We prove for example (R4b): Let $q\in cd$. Then $q=pp'$ where $p\in c$ and $p'\in d$. We have: 
$$_e|(cd) = {}_e|([pp']) = [{}_e|(pp')] =  [{}_e|p \,\, {}_{{\mathbf r}(_e|p)}|p'] = [{}_e|p] [{}_{{\mathbf r}(_e|p)}|p'] = {}_e|[p] \, {}_{{\mathbf r}(_e|p)}|[p'] = {}_e|c \, {}_{{\mathbf r}(_e|c)}|d,$$
as needed. All other items are proved similarly.
\end{proof}

\subsection{From a labelled directed graph with compatible restrictions and corestrictions to a proper Ehresmann semigroup} Let $T$ be a monoid, $E$ a semilattice and ${\mathcal G}$ a labelled directed graph with the vertex set $E$ and edges labelled by elements of $T$ which has compatible restrictions and corestrictions (see Definition \ref{def:labelled_graph}).

Let $E\rtimes_{\mathcal G} T$ be the set which coincides with the underlying set of the category ${{\widetilde{\mathcal C}}}({\mathcal G})$.
On this set we define the operations $\cdot$, \, $^+$ and $^*$ as follows:
\begin{equation}\label{eq:def_product}
\forall c, d\in E\rtimes_{\mathcal G} T\colon \,\, c\cdot d = c|_{{\mathbf r}(c){\mathbf d}(d)} \,\, \space{} _{{\mathbf r}(c){\mathbf d}(d)}|d,
\end{equation}
\begin{equation}\label{eq:def_plus_star}
\forall c\in E\rtimes_{\mathcal G} T\colon \,\, c^+=[({\mathbf d}(c),1,{\mathbf d}(c))] \;\text{ and } \; c^* = [({\mathbf r}(c),1,{\mathbf r}(c))].
\end{equation}

\begin{theorem}\mbox{}\label{th:construction}
\begin{enumerate}
\item \label{i:Ehresmann} $(E\rtimes_{\mathcal G} T,\cdot, ^+, ^*)$ is an Ehresmann semigroup.
\item \label{i:projections} $P(E\rtimes_{\mathcal G} T) = \{[(e,1,e)]\colon e\in E\}$ and $P(E\rtimes_{\mathcal G} T)\simeq E$ 
via the map \mbox{$[(e,1,e)]\mapsto e$.} It follows that $E\rtimes_{\mathcal G} T$ is a monoid if and only if $E$ has a maximum element.
\item \label{i:sigma} Let $c, d\in E\rtimes_{\mathcal G} T$. If $c \mathrel{\sigma} d$ then ${\mathbf l}(c) = {\mathbf l}(d)$. 
\item \label{i:orders}
Let $c,d \in E\rtimes_{\mathcal G} T$. Then $c\leq_l d$ (respectively  $c\leq_r d$) if and only if there is some $e\leq {\mathbf d}(d)$ (respectively $f\leq {\mathbf r}(d)$) such that $c={}_e|d$ (respectively $c=d|_f$). Consequently, $c\leq d$ holds if and only if there are some $e\leq {\mathbf d}(d)$ and $f\leq {\mathbf r}({}_e|d)$ such that $c=({}_e|d)|_f$.
\item \label{i:proper} If $c \mathrel{\sigma} d$ holds if and only if ${\mathbf l}(c) = {\mathbf l}(d)$ for all $c, d\in E\rtimes_{\mathcal G} T$, then  $E\rtimes_{\mathcal G} T$ is proper with the proper generating ideal $Y=\{[(e,s,f)]\colon (e,s,f)\in {\mathrm E}({\mathcal G})\}$ and $(E\rtimes_{\mathcal G} T)/\sigma \simeq T$ as monoids. 
\end{enumerate}
\end{theorem}

Remark that in (4) above and in what follows the symbol $\leq$ denotes both the natural partial order on $E\rtimes_{\mathcal G} T$, and also the order on the semilattice $E$. It is always clear from the context which of these two orders is being used.

\begin{proof}
\eqref{i:Ehresmann} We start from showing that the multiplication $\cdot$ is associative. 
Let $a, b,c\in E\rtimes_{\mathcal G} T$.
Denote $f= {\mathbf r}(a)$, $g={\mathbf d}(b)$,  $h={\mathbf r}(b)$ and $e={\mathbf d}(c)$. 
Let further $m=f{\mathbf d}(b|_{he})$ and $k={\mathbf r}(_{fg}|b)e$. For convenience, in this proof we denote the multiplication in the category ${{\widetilde{\mathcal C}}}({\mathcal G})$ by $\circ$. 
We calculate:
\begin{align*}
(a\cdot b)\cdot c & = (a|_{fg} \circ \space{} _{fg}|b)\cdot c & (\text{by the definition of $\cdot$})\\
& = \Bigl(\bigl(a|_{fg}\circ \space{} _{fg}|b\bigr)|_{k}\Bigr) \circ \space{} _{k}|c & (\text{by the definition of $\cdot$})\\
& = \Bigl((a|_{fg})|_{{\mathbf{d}}\bigl((_{fg}|b)|_{k}\bigr)} \circ (_{fg}|b)|_{k}\Bigr) \circ \space{} _{k}|c & ( \text{by (CR4b)})\\
& = \Bigl(a|_{{\mathbf{d}}\bigl((_{fg}|b)|_{k}\bigr)}  \circ (_{fg}|b)|_{k}\Bigr) \circ \space{} _{k}|c & ( \text{by (CR3b)})\\
& = a|_{{\mathbf{d}}\bigl((_{fg}|b)|_{k}\bigr)} \circ \Bigl((_{fg}|b)|_{k} \circ \space{} _{k}|c \Bigr). & ( \text{since } \circ \text{ is associative})
\end{align*}
Denote $a'= a|_{{\mathbf{d}}\bigl((_{fg}|b)|_{k}\bigr)}$, $b'=(_{fg}|b)|_{k}, c'={}_{k}|c$. On the other hand we have:
\begin{align*}
a\cdot (b\cdot  c) & = a\cdot \bigl( b|_{he}\circ \space{} _{he}|c\bigr) & 
(\text{by the definition of $\cdot$}) \\
&= a|_{m} \circ \space{} _m|\bigl(b|_{he}\circ \space{} _{he}|c\bigr) & (\text{by the definition of $\cdot$})\\
& = a|_{m} \circ \left(\space{} _m|(b|_{he})\circ \space{} _{{\mathbf{r}}(_m|(b|_{he}))}|c\right). & (\text{by (R4b) and (R3b)})
\end{align*}
Denote $a'' = a|_{m}$, $b''={}_m|(b|_{he})$, $c'' = {}_{{\mathbf{r}}(_m|(b|_{he}))}|c$. Observe that
\begin{align*}
b' & = (_{fg}|b)|_{{\mathbf r}(_{fg}|b)e} & \\
& = (_{fg}|b)|_{{\mathbf r}(_{fg}|b)he} & (\text{since } {\mathbf r}(_{fg}|b)\leq h \text{ by (CR1b)})\\
& =  {}_{{\mathbf d}(b|_{he})fg}|(b|_{he})  & (\text{by (Cb)})\\
& =  {}_{{\mathbf d}(b|_{he})f}|(b|_{he}) & (\text{since } {\mathbf d}(b|_{he})\leq g \text{ by (CR1b)}) \\ 
& = b''.
\end{align*}
It follows that  ${\mathbf{d}}((_{fg}|b)|_{k}) = {\mathbf d}(b')={\mathbf d}(b'') = m$ 
and similarly ${{\mathbf{r}}(_m|(b|_{he}))} = k$. 
Hence $a'= a|_{{\mathbf{d}}\bigl((_{fg}|b)|_{k}\bigr)} = a|_m = a''$  and similarly $c'=c''$.  
Thus $(a\cdot b)\cdot c = a'\circ (b'\circ c') = a''\circ (b''\circ c'') = a\cdot (b\cdot c)$.

Let $a\in E\rtimes_{\mathcal G} T$. Applying the definition of $\cdot$, (R2b) and (CR2b) we have 
$$a\cdot a^* = a\cdot [({\mathbf r}(a),1,{\mathbf r}(a))] = a|_{{\mathbf r}(a)} \circ 
{}_{{\mathbf r}(a)}|[({\mathbf r}(a),1,{\mathbf r}(a))] = a \circ [({\mathbf r}(a),1,{\mathbf r}(a))] = a.$$ 
Hence $a\cdot a^* = a$. 
 
Let $a,b\in E\rtimes_{\mathcal G} T$. Put $f={\mathbf r}(a)$, $g={\mathbf r}(b)$. Then applying (R5b) and (CR5b) we have 
\begin{multline*}a^*\cdot b^* = [(f,1,f)] \cdot [(g,1,g)] = [(f,1,f)]|_{fg}\circ \space{} _{fg}|[(g,1,g)] = \\ [(fg,1,fg)]\circ [(fg,1,fg)] = [(fg,1,fg)^2] = [(fg,1,fg)].
\end{multline*}
It follows that $(a^*\cdot b^*)^*=a^*\cdot b^* = b^*\cdot a^*$ holds.

Let $a, b\in E\rtimes_{\mathcal G} T$. Put $f={\mathbf r}(a)$, $g={\mathbf d}(b)$. Then $(a\cdot b)^* =  (a|_{fg} \circ\space{} _{fg}|b)^* = [({\mathbf r}(_{fg}|b),1,{\mathbf r}(_{fg}|b))]$ 
and $(a^*\cdot b)^* = ([(f,1,f)]|_{fg} \circ {} _{fg}|b)^* =  [({\mathbf r}(_{fg}|b),1,{\mathbf r}(_{fg}|b))]$. It follows that $(a\cdot b)^* = (a^*\cdot b)^*$.

We have verified that $(E\rtimes_{\mathcal G} T, \cdot, ^*)$ is a right Ehresmann semigroup. Dually it follows that $(T\rtimes_{\mathcal G} E, \cdot, ^+)$ is a left Ehresmann semigroup. In addition, it is easy to see that $(a^*)^+ = a^*$ and $(a^+)^* = a^+$ for all $a\in E\rtimes_{\mathcal G} T$. 
Therefore, $(E\rtimes_{\mathcal G} T, \cdot, ^+, ^*)$ is an Ehresmann semigroup.

\eqref{i:projections} Let $a\in E\rtimes_{\mathcal G} T$ be a projection. Then $a=a^+=[({\bf d}(a),1, {\bf d}(a))]$ and also clearly any $[(e,1,e)]$, $e\in E$, is a projection. Now, since $[(e,1,e)]\cdot [(f,1,f)] = [(ef,1,ef)]$, it follows that $P(E\rtimes_{\mathcal G} T)$ is a semilattice and the map $[(e,1,e)]\to e$ is an isomorphism of  semilattices. By Lemma \ref{lem:monoid} $E\rtimes_{\mathcal G} T$ is a monoid if and only if it has a maximum projection. As we have shown, this is the case if and only if $E$ has a maximum element.

\eqref{i:sigma} Let  $\varphi\colon  E\rtimes_{\mathcal G} T\to T$ be the map given by $a\mapsto {\mathbf l}(a)$ where $T$ is a reduced Ehresmann semigroup. It is easy to see that is a $(2,1,1)$-morphism.
Note that $\varphi$ is surjective, because  for every $t\in T$ there is a path in ${\mathcal G}$ labelled by $t$. So $T$ is a reduced quotient of $E\rtimes_{\mathcal G} T$. 
Hence $T$ is a quotient of $(E\rtimes_{\mathcal G} T)/\sigma$. 
It follows  that $a \mathrel{\sigma} b$ implies that $\varphi(a)=\varphi(b)$, that is, ${\mathbf l}(a) = {\mathbf l}(b)$, which proves \eqref{i:sigma}.

\eqref{i:orders} 
Let $c,d \in E\rtimes_{\mathcal G} T$ be such that $c\leq_l d$. This means that there is a projection $[(e,1,e)]$ such that $c=[(e,1,e)]\cdot d$. Applying \eqref{eq:def_product} and (CR5b), this is equivalent to $c =[(e{\bf d}(d),1,e{\bf d}(d))] \circ {}_{e{\bf d}(d)}|d = {}_{e{\bf d}(d)}|d$. The statement for
$\leq_r$ is dual, and the statement for $\leq$ follows from those for $\leq_l$ and $\leq_r$.

\eqref{i:proper} Suppose that $c \mathrel{\sigma} d$ holds if and only if ${\mathbf l}(c) = {\mathbf l}(d)$ for all $c, d\in E\rtimes_{\mathcal G} T$. 
It follows from \eqref{i:orders} that $Y=\{[(e,s,f)]\colon (e,s,f)\in {\mathrm E}({\mathcal G})\}$ is an order ideal of $E\rtimes_{\mathcal G} T$ with respect to the natural partial order. 

If $c= [(e,s,f)], d=[(g,t,h)]\in Y$ are such that $c^+=d^+$, $c^*=d^*$ and $c\,\sigma\,d$ then, clearly, 
$s = {\mathbf l}(c) = {\mathbf l}(d) = t$ and also $e=g$ and $f= h$, so $c=d$. It follows that all elements of $Y$ are proper. In addition, $Y$ contains all the projections of $E\rtimes_{\mathcal G} T$, by the definition of ${\mathcal G}$. 

Now each element of $E\rtimes_{\mathcal G} T$ decomposes as a matching product of elements of $Y$ and, moreover, the definition of $\sim$ yields that such a factorization is unique up to equivalence. Therefore, the Ehresmann semigroup $E\rtimes_{\mathcal G} T$ is proper. By the assumption, the map $(E\rtimes_{\mathcal G} T)/\sigma \to T$ given by  $[a]_\sigma \mapsto {\mathbf l}(a)$ is well defined and injective. Since for every $t\in T$ there is a path in ${\mathcal G}$ labelled by $t$, this map is surjective. It is also a monoid homomorphism,  and thus an isomorphism. Hence $(E\rtimes_{\mathcal G} T)/\sigma \simeq T$. 
\end{proof}

In what follows the multiplication in $E\rtimes_{\mathcal G} T$ will be often denoted by juxtaposition.

Since $E\rtimes_{\mathcal G} T$ and ${{\widetilde{\mathcal C}}}({\mathcal G})$ have the same underlying set, the relations $\leq_l$, $\leq_r$ and $\leq$ are also partial orders on the category ${{\widetilde{\mathcal C}}}({\mathcal G})$. It is routine to show that, equipped with $\leq_l$ and $\leq_r$, the category ${{\widetilde{\mathcal C}}}({\mathcal G})$ is an Ehresmann category in the sense of Lawson \cite{Lawson91}. After this, the fact that $E\rtimes_{\mathcal G} T$ is an Ehresmann semigroup becomes also a consequence of \cite[Theorem 4.21]{Lawson91}.

Observe that the partial orders $\leq_l$ and $\leq_r$ can be defined already on ${\mathrm{E}}({\mathcal G})$ by $u \leq_l v$ 
if there is some $e\leq {\mathbf d}(v)$  such that $u={}_e|v$ and $u\leq_r v$ if there is some $f\leq {\mathbf r}(v)$ such that $u=v|_f$. Using (C) one can see that   $\leq_l \circ \leq_r \,=\,\, \leq_r \circ \leq_l$, denote this relation by $\leq$. Clearly, $u\leq_l v$ (respectively $u\leq_r v$, or $u\leq v$) in $E(\mathcal{G})$ implies $[u]\leq [v]$ (respectively $[u]\leq_r [v]$, or $[u]\leq [v]$) in $\widetilde{\mathcal{C}}(\mathcal{G})$.

We now single out a special case for which the assumption of Theorem \ref{th:construction}\eqref{i:proper} holds.

\begin{proposition}\label{prop:proper_special} Suppose that $T=X^*$ is the free $X$-generated monoid, edges of ${\mathcal G}$ are labelled by elements of $X$ or $1$, where the edges labelled by $1$ are precisely the edges $(e,1,e)$. Suppose further that for any two edges $u,v\in {\mathrm E}({\mathcal G})$ with ${\mathbf l}(u) = {\mathbf l}(v)\neq 1$ there is an edge $w\in  {\mathrm E}({\mathcal G})$ such that $u,v\leq w$. 
Then for all $c,d\in E\rtimes_{\mathcal G} T$: $c \mathrel{\sigma} d$ if and only if ${\mathbf l}(c) = {\mathbf l}(d)$. Consequently, the Ehresmann semigroup $E\rtimes_{\mathcal G} T$ is proper with  the proper generating ideal $Y=\{[(e,s,f)]\colon (e,s,f)\in {\mathrm E}({\mathcal G})\}$ and $(E\rtimes_{\mathcal G} T)/\sigma \simeq T$.
\end{proposition}

\begin{proof} In view of Theorem \ref{th:construction}\eqref{i:sigma} 
it is enough to assume that ${\mathbf l}(c) = {\mathbf l}(d)$ and to show that $c \mathrel{\sigma} d$. (We note that  in view of Definition \ref{def:labelled_graph} each element of $X$ labels at least one edge of ${\mathcal G}$.) Let ${\mathbf l}(c) = {\mathbf l}(d) = a_1\cdots a_n$, where $a_i\in X$ for $i=1,\dots, n$ and \mbox{$n\geq 1$.} Then $c=[(e_0,a_1,e_1,\dots, e_{n-1}, a_n, e_n)]$ and $d=[(f_0,a_1,f_1,\dots, f_{n-1}, a_n,f_n)]$ for some $e_i, f_i\in E$, $0\leq i\leq n$. For each $i=0,\dots, n-1$ let $(g_i,a_{i+1},g_{i+1})\in {\mathrm{E}}({\mathcal G})$ be such that $(e_i,a_{i+1},e_{i+1}), (f_i,a_{i+1},f_{i+1})\leq (g_i,a_{i+1},g_{i+1})$. Then $[(e_i,a_{i+1},e_{i+1})] \mathrel{\sigma} [(g_i,a_{i+1},g_{i+1})] \mathrel{\sigma} [(f_i,a_{i+1},f_{i+1})]$. This yields $c\mathrel{\sigma} d$. If ${\mathbf l}(c)={\mathbf l}(d) =1$, then $c \mathrel{\sigma}  d$ as $\sigma$ identifies all the projections. 
\end{proof}

\subsection{The structure theorem} Let $S$ be a proper Ehresmann semigroup and $Y$ be a proper generating ideal of $S$ (see Definition \ref{def:proper}). 
Denote $T=S/\sigma$ and recall that for $a\in S$ by  $[a]_{\sigma}\in T$ we denote the $\sigma$-class of $a$. The {\em underlying graph} ${\mathcal G}_{S,Y}$ of $S$ with respect to $Y$ is the labelled directed graph defined as follows. Vertices of ${\mathcal G}_{S,Y}$ are elements of $P(S)$ and for all $e,f\in P(S)$ edges of  ${\mathcal G}_{S,Y}$ from $e$ to $f$ are triples $(e,t,f)$ where $t\in T$ and there is a (necessarily unique) $a\in Y$ such that $e=a^+$, $f=a^*$ and $t=[a]_\sigma$. It is immediate that 
$a\mapsto (a^+,[a]_{\sigma},a^*)$ is a bijection between  $Y$ and edges of ${\mathcal G}_{S,Y}$.

Let $p=(e,t,f)\in {\mathrm E}({\mathcal G}_{S,Y})$ and let $a\in Y$ be such that $p=(a^+,[a]_{\sigma},a^*)$. For  a projection
$g\leq e$ we define the {\em restriction} $_g|p$ of $p$ to $g$ by 
\begin{equation}\label{eq:restriction}
_g|p = ((ga)^+, [ga]_{\sigma}, (ga)^*) = (g, [a]_{\sigma}, (ga)^*).
\end{equation} 
Similarly, for $h\leq f$ we define the {\em corestriction} $p|_h$ of $p$ to $h$ 
\begin{equation}\label{eq:corestriction}
 p|_h = ((ah)^+, [ah]_{\sigma}, (ah)^*) = ((ah)^+, [a]_{\sigma}, h).
\end{equation} 
These are well defined since $ga, ah\leq a$ and $Y$ is an order ideal.
\begin{lemma}\label{lem:graph}
${\mathcal G}_{S,Y}$ is a labelled directed graph with compatible restrictions and corestrictions.
\end{lemma}

\begin{proof} Since $P(S)\subseteq Y$ and by the definition of ${\mathcal G}_{S,Y}$, we have that for every $e\in P(S)$ the graph ${\mathcal G}_{S,Y}$ has an edge $(e,1,e)$. In addition, for each $t\in S/\sigma$ the graph ${\mathcal G}_{S,Y}$ has a path labelled by $t$.

Let $p=(a^+,[a]_{\sigma},a^*)\in {\mathrm E}({\mathcal G}_{S,Y})$ and $g\leq a^+$. Then $_g|p =  (g, [a]_{\sigma}, (ga)^*)$ and since $(ga)^*a^*=(gaa^*)^* = (ga)^*$, we have $(ga)^*\leq a^*$ so that (R1) holds.

Let $p=(a^+,[a]_{\sigma},a^*)\in {\mathrm E}({\mathcal G}_{S,Y})$. Then $_{a^+}|p = (a^+, [a]_{\sigma}, (a^+a)^*)= p$, so that (R2) holds.

Let us verify that condition (R3) holds. Let $h\leq g\leq a^+$. Then
$$
_h|(_g|p) = \space{} _h|((ga)^+,[ga]_{\sigma},(ga)^*) = ((hga)^+, [hga]_{\sigma}, (hga)^*) = ((ha)^+, [ha]_{\sigma}, (ha)^*) = \space{}_h|p,
$$
as needed.

To show (R4) let  $p_1=(a_1^+,[a_1]_{\sigma},a_1^*), \dots, p_n=(a_n^+,[a_n]_{\sigma},a_n^*) \in {\mathrm E}({\mathcal G}_{S,Y})$ where $a_i^*=a_{i+1}^+$ for all $i=1,\dots, n-1$ be such that $(a_1^+,[a_1\cdots a_n]_{\sigma}, a_n^*) \in  {\mathrm E}({\mathcal G}_{S,Y})$ where $n\geq 2$.  Put $e_0 = a_1^+$ and let  $e_0'\leq e_0$. 
Define 
$e_{1}' = {\mathbf r}(_{e_0'}|p_1)=(e_0'a_1)^*$, $e_{2}' = {\mathbf r}(_{e_{1}'}|p_{2})=((e_0'a_1)^*a_2)^* = (e_0'a_1a_2)^*$, $\dots,$ $e_n' = {\mathbf r}(_{e_{n-1}'}|p_n)=((e_0'a_1\cdots a_{n-1})^*a_n)^* = (e_0'a_1\cdots a_{n-1}a_n)^*$. Now  
$$_{e_0'}|(a_1^+,[a_1\cdots a_n]_{\sigma}, a_n^*)=(e_0', [a_1]_{\sigma}\cdots [a_n]_{\sigma}, (e_0'a_1\cdots a_{n-1}a_n)^*) = (e_0', [a_1]_{\sigma}\cdots [a_n]_{\sigma}, e_n'),$$ 
as needed.

If $e,f\in P(S)$ are such that $f\leq e$ then we have:
$$_f|(e,1,e) = {}_f|(e^+,[e]_{\sigma},e^*) = (f,[e]_{\sigma},(ef)^*) = (f,1,f),$$
so that (R5) holds.  

Axioms (CR1)--(CR5) follow dually.

We finally verify axiom (C). Let $p=(a^+,[a]_{\sigma},a^*)\in {\mathcal G}_{S,Y}$ and $g\leq a^+$, $h\leq a^*$. 
We calculate: 
\begin{align*}
(_g|p)|_{{\mathbf r}(_g|p)h} & = ((ga)^+, [a]_{\sigma}, (ga)^*)|_{(ga)^*h}  & (\text{by \eqref{eq:restriction}})\\
& = ((ga(ga)^*h)^+, [a]_{\sigma}, (ga)^*h) & (\text{by \eqref{eq:corestriction}})\\
& = ((gah)^+, [a]_{\sigma}, (ga)^*h) & (\text{by \eqref{eq:axioms_star}}) \\ 
& = (g(ah)^+, [a]_{\sigma}, (ga)^*h) &  (\text{by \eqref{eq:rule1}}) \\
& = (g{\mathbf d}(p|_h), [a]_{\sigma}, {\mathbf r}(_g|p)h). & 
(\text{by \eqref{eq:restriction} and \eqref{eq:corestriction})}
\end{align*}  
The equality  $_{{\mathbf d}(p|_h)g}|(p|_h) = (g{\mathbf d}(p|_h), [a]_{\sigma}, {\mathbf r}(_g|p)h)$ follows dually.
\end{proof}

Therefore, we can construct the Ehresmann semigroup $P(S) \rtimes_{{\mathcal G}_{S,Y}} S/\sigma$. 

The following theorem describes the structure of proper Ehresmann semigroups in terms of labelled directed graphs with compatible restrictions and corestrictions. We will show in Section~\ref{s:multiactions} that this result generalizes the structure result on proper restriction semigroups in terms of partial actions \cite{CG12, K15} and is thus a wide-ranging extension of the McAlister theorem on the structure of $E$-unitary inverse semigroups, formulated in terms of partial actions \cite{PR79, KL04}.

\begin{theorem} \label{th:main}
Let $S$ be a proper Ehresmann semigroup and let  $Y$ be its proper generating ideal. Then $S\simeq P(S) \rtimes_{{\mathcal G}_{S,Y}} S/\sigma$. 
\end{theorem}

\begin{proof}  Let $a\in S$  and let $(a_1,\dots, a_n)$ be a matching factorization of $a$ into a product of elements of $Y$. We define $(a_1,\dots, a_n)\overline{\Psi}= (a_1^+, [a_1]_{\sigma}, a_2^+, [a_2]_{\sigma}, \dots a_{n}^+, [a_n]_{\sigma}, a_n^*) \in {\mathcal C}({\mathcal G}_{S,Y})$. Suppose that  $t=a_ia_{i+1}\cdots a_j\in Y$ where $1\leq i<j \leq n$ and let $(a_1,\dots, a_{i-1},t, a_{j+1},\dots, a_n)$ be the matching factorization of $a$ into a product of elements of $Y$ obtained by the replacement of $(a_i,a_{i+1},\dots, a_j)$ with $t$. The definition of the congruence $\sim$ on ${\mathcal C}({\mathcal G}_{S,Y})$ yields that $(a_1,\dots, a_{i-1},t,a_{j+1},\dots, a_n)\overline{\Psi} \sim (a_1,\dots, a_n)\overline{\Psi}$. It follows that $\overline{\Psi}$ maps equivalent matching factorizations of $a$ into products of elements of $Y$ to equivalent paths in  ${\mathcal C}({\mathcal G}_{S,Y})$. Therefore, we can define the map $\Psi\colon S\to  P(S) \rtimes_{{\mathcal G}_{S,Y}} S/\sigma$ by 
$$a\Psi = [(a_1^+, [a_1]_{\sigma}, a_2^+, [a_2]_{\sigma}, \dots, a_{n}^+, [a_n]_{\sigma}, a_n^*)],$$
where $(a_1,\dots, a_n)$ is a matching factorization of $a$ into a product of elements of $Y$.

It is immediate from the definitions that $\Psi$ is a bijection. Let us show that it preserves the multiplication. Let $a,b\in S$ and let $(a_1,\dots, a_n)$ and $(b_1,\dots, b_m)$ be matching factorizations of $a$ and $b$ into products of elements of $Y$. Let $a_{n}' = a_nb_1^+$, $a_{n-1}' = a_{n-1}(a_{n}')^+$, $\dots$,  $a_1' = a_1(a_{2}')^+$ and $b_1' = a_n^* b_1$, $b_2' = (b_1')^*\,b_2$, $\dots$, $b_m' = (b_{m-1}')^* \,b_m$. Note that $(a'_n)^* = (a_nb_1^+)^* = a_n^*b_1^+ = (a_n^*b_1)^+ = (b'_1)^+$. By  Lemma~\ref{lem:matching2} and its dual we see that  
$(a_1',\dots, a_n',b_1',\dots, b_m')$ is a matching factorization of $ab$ into a product elements of $Y$, thus  
$$
(ab)\Psi = [((a_1')^+, [a_1]_{\sigma},  \dots, (a_{n}')^+, [a_n]_{\sigma}, (a_n')^*, [b_1]_{\sigma}, (b_1')^*, \dots, [b_m]_{\sigma}, (b_m')^*)].
$$
From \eqref{eq:def_product} we have that
$$
(a\Psi) (b\Psi)  = [(a_1^+, [a_1]_{\sigma},  \dots, a_{n}^+,[a_n]_{\sigma}, a_n^*)]|_{a_n^*b_1^+} {}_{a_n^*b_1^+}| [(b_1^+, [b_1]_{\sigma},\dots, b_m^+, [b_m]_{\sigma}, b_m^*)] ,
$$
which, in view of \eqref{eq:restriction}, (RPath) and (CRPath),
equals
$$
 [((a_1')^+, [a_1]_{\sigma},  \dots, (a_{n}')^+, [a_n]_{\sigma}, (a_n')^*][((b_1')^+, [b_1]_{\sigma}, (b_1')^*, \dots, [b_m]_{\sigma}, (b_m')^*)] = (ab)\Psi.
$$
Let $(a_1,\dots, a_n)$ be a matching factorization of $a$ into a product of elements of $Y$. Then
\begin{equation*}
(a^+)\Psi = (a_1^+)\Psi = [(a_1^+,[a_1^+]_{\sigma},a_1^+)] 
= [(a_1^+, [a_1]_{\sigma}, a_2^+, [a_2], \dots, [a_n]_{\sigma}, a_n^*)]^+ = 
(a\Psi)^+.
\end{equation*}
It follows that $\Psi$ preserves the operation $+$. Dually it also preserves $*$. Hence $\Psi$ is a  $(2,1,1)$-isomorphism. 
\end{proof}

Combining Theorem \ref{th:construction} and Theorem \ref{th:main}, we obtain the following result.
\begin{theorem}
Let $T$ be a monoid, $E$ a semilattice and ${\mathcal G}$ a labelled directed graph with the vertex set $E$ and edges labelled by elements of $T$ which has compatible restrictions and corestrictions (see Definition \ref{def:labelled_graph}). Suppose that $c \mathrel{\sigma} d$ holds in $E\rtimes_{\mathcal G} T$ if and only if ${\mathbf l}(c) = {\mathbf l}(d)$ for all $c, d\in E\rtimes_{\mathcal G} T$. Then $E\rtimes_{\mathcal G} T$ is a proper Ehresmann semigroup with the proper generating ideal $Y=\{[(e,s,f)]\colon (e,s,f)\in {\mathrm E}({\mathcal G})\}$ and the semilattice of projections isomorphic to $E$ via $[(e,1,e)]\mapsto e$. In particular, $E\rtimes_{\mathcal G} T$ is a monoid if and only if $E$ has a top element. Furthermore, $(E\rtimes_{\mathcal G} T)/\sigma \simeq T$ as monoids. 

Conversely, every proper Ehresmann semigroup $S$ has this form (up to isomorphism). Specifically, if $Y$ is a proper generating ideal of $S$ then  $S\simeq P(S) \rtimes_{{\mathcal G}_{S,Y}} S/\sigma$. 
\end{theorem}

\section{Covers}\label{s:covers} 

An Ehresmann semigroup $T$ is called a {\em cover} of an Ehresmann semigroup $S$, if there is a surjective projection-separating morphism $\varphi\colon T\to S$. Let $X^+ = X^*\setminus \{1\}$ be the free $X$-generated semigroup.

\begin{theorem}\label{th:cover} Any Ehresmann semigroup has a proper cover.
\end{theorem}

\begin{proof} Let $S$ be an Ehresmann semigroup and let $X$ be a $(2,1,1)$-generating set of $S$. For $v\in X^+$ let $\overline{v}$ be the value of $v$ in $S$. 
We construct a labelled directed graph ${\mathcal{G}}$ with vertex set $P(S)$ and edges labelled by elements of $X\cup\{1\}$ as follows. Edges labelled by $1$ are precisely the edges $(e,1,e)$ where $e\in P(S)$ and for $a\in X$ the edges labelled by $a$ are all the edges $(e,a,f)$ where $e,f\in P(S)$ are such that $e=(e\overline{a}f)^+$, $f= (e\overline{a}f)^*$.
We define the restriction and corestriction on ${\mathcal G}$ as follows. For an edge $(e,a,f)$ and $g\leq e$ we put
\begin{equation}\label{eq:restr_cover}
_g|(e,a,f) = (g, a, (g\overline{a})^*f) = ((g\overline{a}f)^+, a, (g\overline{a}f)^*)
\end{equation}
and for $h\leq f$ we put
\begin{equation}\label{eq:corestr_cover}
(e,a,f)|_h = (e(\overline{a}h)^+, a, h) = ((e\overline{a}h)^+, a, (e\overline{a}h)^*).
\end{equation}

The second equality in \eqref{eq:restr_cover} holds because 
\begin{align*}
(g\overline{a}f)^+ & = (ge\overline{a}f)^+ & (\text{since } g\leq e)\\
& = g(e\overline{a}f)^+  & (\text{by } \eqref{eq:rule1})\\
& = ge = g & (\text{since } (e\overline{a}f)^+=e \text{ and } g\leq e)
\end{align*}
and $(g\overline{a})^*f = (g\overline{a}f)^*$ by \eqref{eq:rule1}.
Observe that $_g|(e,a,f)$ is an edge since $(g\overline{a}(g\overline{a})^*f)^+ = (g\overline{a}f)^+ = g$ and $(g\overline{a}(g\overline{a})^*f)^* = (g\overline{a}f)^*$. Similarly, one shows that the second equality in \eqref{eq:corestr_cover} holds and that $(e,a,f)|_h$ is an edge.

Furthermore, if $f\leq e$ we put $(e,1,e)|_f = \space{} _f|(e,1,e)= (f,1,f)$.

Let us show that ${\mathcal G}$ is a labelled directed graph with compatible restrictions and corestrictions.
It is immediate that axioms (R1), (R2) and (R5) hold. 

To show that (R3) holds, we need to show that $_h|(_g|(e,a,f)) = {} _h|(e,a,f)$ whenever $h\leq g\leq e$. This reduces to showing that $(h,a, (h\overline{a})^*(g\overline{a})^*f) = (h,a, (h\overline{a})^*f)$. It is enough to verify the equality  $(h\overline{a})^*(g\overline{a})^*= (h\overline{a})^*$. We have: 
\begin{align*}
(h\overline{a})^*(g\overline{a})^* & = (h\overline{a}(g\overline{a})^*)^* & (\text{by } \eqref{eq:rule1})\\
& = (hg\overline{a}(g\overline{a})^*)^* & (\text{since } h\leq g)\\
& = (hg\overline{a})^* & (\text{by } \eqref{eq:axioms_star})\\
& = (h\overline{a})^*, &  (\text{since } h\leq g)
\end{align*}
as needed.

To show (R4), observe that if $p_1=(e_0,t_1,e_1), \dots, p_n=(e_{n-1},t_n,e_n) \in {\mathrm E}({\mathcal G})$ with $t_i\in X\cup \{1\}$ such that $(e_0,t_1\cdots t_n,e_n) \in  {\mathrm E}({\mathcal G})$ then at most one of the elements $t_i$ differs from $1$, since labels of edges of ${\mathcal G}$ belong to $X\cup \{1\}$. If $t_i=1$  for all $i$ then $e_i=e_0$ and $p_i=(e_0,1,e_0)$ for all  $i$. Hence $(e_0,t_1\cdots t_n,e_n) = (e_0, 1, e_0)$ and thus (R4) obviously holds. If $t_k\in X$ for some $k\in \{1,\dots, n\}$ and $t_i=1$ for all $i\neq k$ then $(e_0,t_1\cdots t_n,e_n)=(e_{k-1}, t_k, e_k)=p_k$, $p_i=(e_{k-1},1, e_{k-1})$ for $i<k$ and $p_i=(e_k,1,e_k)$ for $i>k$. It is now easy to see that (R4) holds in this case, too.  By symmetry, axioms (CR1)--(CR5) also hold.

We finally check that axiom (C) holds. Let $c=(e,a,f)\in {\mathrm{E}}(\mathcal G)$ where $a\in X$ and let $g\leq e$, $h\leq f$. Then $_g|c = (g,a, (g\overline{a})^*f)$. Thus  ${\mathbf r}(_{g}|c)h = (g\overline{a})^*fh = (g\overline{a})^*h$ which, in view of \eqref{eq:rule1}, equals $(g\overline{a}h)^*$. Similarly one checks that ${\mathbf d}(c|_h)g = (g\overline{a}h)^+$. Then:
\begin{align*}
(_{g}|c)|_{{\mathbf r}(_{g}|c)h} & =  (g,a,(g\overline{a})^*f)|_{(g\overline{a}h)^*}  &  (\text{by } \eqref{eq:restr_cover})\\
& = (g(\overline{a}(g\overline{a})^*h)^+,a,(g\overline{a}h)^*) & (\text{by } \eqref{eq:corestr_cover})\\
& = ((g\overline{a}(g\overline{a})^*h)^+,a,(g\overline{a}h)^*) & (\text{by } \eqref{eq:rule1})\\
& = ((g\overline{a}h)^+,a,(g\overline{a}h)^*) & (\text{by } \eqref{eq:axioms_star})
\end{align*}
and symmetrically $_{{\mathbf d}(c|_h)g}|(c|_h) = ((g\overline{a}h)^+,a,(g\overline{a}h)^*)$. If $a=1$, then $e=f$ and it is easy to see that both $(_{g}|c)|_{{\mathbf r}(_{g}|c)h}$ and  $_{{\mathbf d}(c|_h)g}|(c|_h)$ are equal to $(gh,1,gh)$. Thus (C) holds. 

We can thus form the Ehresmann semigroup $P(S)\rtimes_{\mathcal G} X^*$. Let us show that
it is proper. Let $(e,a,f)\in {\mathrm E}({\mathcal G})$. Then $e=(e\overline{a}f)^+ = e(\overline{a}f)^+\leq (\overline{a}f)^+ \leq \overline{a}^+$. 
It follows that $_e|(\overline{a}^+, a, \overline{a}^*)$ is defined and, in view of \eqref{eq:restr_cover} and $(e\overline{a})^*\overline{a}^* = (e\overline{a}\,\overline{a}^*)^* = (e\overline{a})^*$, it equals $(e, a, (e\overline{a})^*)$. Now, since $f=(e\overline{a}f)^* = (e\overline{a})^*f\leq (e\overline{a})^*$, it follows that the element $(e, a, (e\overline{a})^*)|_f$ is defined and, in view of \eqref{eq:corestr_cover}, it equals $(e,a,f)$. It follows that $(e,a,f)\leq (\overline{a}^+, a, \overline{a}^*)$. By Proposition~\ref{prop:proper_special} we have that the Ehresmann semigroup $P(S)\rtimes_{\mathcal G} X^*$ is proper and $(P(S)\rtimes_{\mathcal G} X^*)/\sigma \mathrel{\simeq} X^*$.

Define a map $\varphi\colon P(S)\rtimes_{\mathcal G} X^*\to S$ by 
\begin{align*}
& [(e,1,e)]\varphi=e,\\
&[(e_0, a_1,\dots, a_n,e_n)] \varphi = e_0\overline{a_1}e_1 \dots \overline{a_n}e_n.
\end{align*}
The map $\varphi$ is well-defined because all the paths, which are equivalent to $(e_0, a_1,\dots, a_n,e_n)$, are obtained from it by inserting or removing edges of type $(e,1,e)$ (and no other edges).
Let us show that $\varphi$ preserves the multiplication. It is immediate that $[(e,1,e)]\varphi[(f,1,f)]\varphi = ef = [(ef,1,ef)]\varphi = ([(e,1,e)][(f,1,f)])\varphi$. 

Let now $s = [(e_0, a_1,\dots, a_n,e_n)]$ and $t=[(f_0,b_1,\dots, b_k,f_k)]$ where $n,k\geq 1$ and $a_i, b_j\in X$ for all $i=1,\dots, n$ and $j=1,\dots, k$. Then by (RPath), (CRPath), (4.1) and (4.2) we have
$$
s|_{e_nf_0} = [(h_0,a_1,h_1,a_2,\ldots,a_n,h_n)],
$$
where $h_n=e_nf_0$ and $h_i = e_i(\overline{a_{i+1}} h_{i+1})^+$ for $i=0,\ldots,n-1$, and 
$$
_{e_nf_0}|t = [(g_0,b_1,g_1,b_2,\ldots,b_k,g_k)],
$$
where $g_0 = e_nf_0 = h_n$ and $g_i= (g_{i-1}\overline{b_i})^* f_i$ for $i=1,\ldots,k$. Applying \eqref{eq:def_product} we calculate:
\begin{align*}
(st)\varphi & = 
(s|_{e_nf_0} \, _{e_nf_0}|t) \varphi\\ 
& = [(h_0,a_1,\ldots,a_n,h_n,b_1,g_1,\ldots,b_k,g_k)]\varphi &  \\ 
& = h_0\overline{a_1}h_1\overline{a_2}\cdots \overline{a_n}e_nf_0\overline{b_1}g_1 \cdots \overline{b_k}g_k & (\text{by the construction of } \varphi)\\
& = e_0(\overline{a_1}h_1)^+\overline{a_1}h_1\overline{a_2}\cdots \overline{a_n}e_nf_0\overline{b_1}g_1 \cdots \overline{b_k}g_k & (\text{since } h_0= e_0(\overline{a_1}h_1)^+)
\end{align*}
\begin{align*}
& = e_0\overline{a_1}h_1\overline{a_2}h_2\cdots \overline{a_n}e_nf_0\overline{b_1}g_1 \cdots \overline{b_k}g_k & (\text{since } (\overline{a_1}h_1)^+\overline{a_1}h_1 = \overline{a_1}h_1)\\
& = e_0\overline{a_1}e_1 (\overline{a_2}h_2)^+\overline{a_2}h_2\cdots \overline{a_n}e_nf_0\overline{b_1}g_1\cdots\overline{b_k}g_k &  (\text{since } h_1= e_1(\overline{a_2}h_2)^+)\\
& \dots \\ 
& = e_0\overline{a_1}e_1\overline{a_2}e_2\cdots \overline{a_n}e_nf_0\overline{b_1}g_1 \cdots \overline{b_k}g_k.& 
\end{align*}
Starting from $g_k$ and moving leftwards we then similarly arrive at 
$$
(st)\varphi =  e_0\overline{a_1}e_1\overline{a_2}e_2\cdots \overline{a_n}e_n f_0\overline{b_1}f_1 \cdots \overline{b_k}f_k,
$$
which is precisely $s\varphi t\varphi$.

The remaining two cases where $s = [(e_0, a_1,\dots, a_n,e_n)]$ with $a_i\in X$ for all $i=1,\dots, n$ and $t=[(f,1,f)]$, and the dual case, are treated similarly. 

Let us show that $\varphi$ preserves the unary operation $^+$. If $s=[(e,1,e)]$ it is immediate that $(s^+)\varphi = (s\varphi)^+=e$. Let $s = [(e_0, a_1,\dots, a_n,e_n)]$ with $a_i\in X$ for all $i=1,\dots, n$. Then $(s^+)\varphi = [(e_0,1,e_0)]\varphi = e_0$. On the other hand, we have:
\begin{align*}
(s\varphi)^+ & = 
(e_0\overline{a_1}e_1\cdots e_{n-1}\overline{a_n}e_n)^+ & (\text{by the definition of } \varphi) \\ 
& = (e_0\overline{a_1}e_1\cdots e_{n-2}\overline{a_{n-1}} (e_{n-1}\overline{a_n}e_n)^+)^+ & (\text{applying the third identity of } \eqref{eq:axioms_plus}) \\ 
& = (e_0\overline{a_1}e_1\cdots e_{n-2}\overline{a_{n-1}}e_{n-1})^+ & (\text{by the definition of } {\mathcal{G}}) \\ 
& \dots & \\ 
& = (e_0\overline{a_1}e_1)^+= e_0.   &
\end{align*}
Hence $(s\varphi)^+=(s^+)\varphi$, as needed. By a dual argument, $\varphi$ preserves also the  operation $^*$.

Since $P(P(S)\rtimes_{\mathcal{G}} X^*) = \{[(e,1,e)]\colon e\in P(S)\}$ and $[(e,1,e)] \varphi = e$, it is immediate that $\varphi$ is projection separating.

Let us show that $\varphi$ is surjective. It suffices to show that every element $s\in S$ can be written as a product $s=e_0\overline{x_1}e_1 \cdots \overline{x_n}e_n$ where $e_i\in P(S)$,
$x_i\in X$ and for all $i=1,\dots n:$ $(e_{i-1}\overline{x_i}e_i)^+=e_{i-1}$, $(e_{i-1}\overline{x_i}e_i)^*=e_i$. Since $S$ is $(2,1,1)$-generated by $X$, every element of $S$ can be written as a product of projections and elements of the multiplicative subsemigroup of $S$ generated by $X$. Since, for $x,y\in X$ we have $\overline{xy} = (\overline{x}^+ \, \overline{x} \, \overline{x}^*\overline{y}^+) (\overline{x}^*\overline{y}^+ \, \overline{y} \, \overline{y}^*)$,  any $s\in S$ can be written as $s=(f_0\overline{x_1}f_1) (f_1\overline{x_2}f_2) \cdots (f_{n-1}\overline{x_n}f_n)$ where $f_i\in P(S)$ and $x_i\in X$. Lemma \ref{lem:matching3} now implies that there are $s_i\leq f_{i-1}\overline{x_i}f_i$, $i=1,\dots, n$, such that $(s_1,\dots, s_n)$ is a matching factorization of $s$. Let $i\in \{1,\dots, n\}$ and note that $s_i = g\overline{x_i}h$ for some $g,h\in P(S)$. Put $e_{i-1} = g(\overline{x_i}h)^+$ and $e_i = (g\overline{x_i})^*h$. We then can write $s_i = e_{i-1} \overline{x_i}e_i$ where $s_i^+ =  (g\overline{x_i}h)^+ =  g(\overline{x_i}h)^+ = e_{i-1}$ and  $s_i^* = (g\overline{x_i})^*h = e_i$. Hence $s= e_0\overline{x_1}e_1 \cdots \overline{x_n}e_n$ is the required factorization of $s$.
\end{proof}

\begin{remark} If a proper Ehresmann semigroup $S$ has the identity element then by Lemma~\ref{lem:monoid} it coincides with the maximum projection $1_{P(S)}$. Then the proper Ehresmann semigroup $P(S)\rtimes_{\mathcal G} X^*$ has the identity element $[(1_{P(S)}, 1, 1_{P(S)})]$. The definition of the covering map $\varphi$ in the proof of Theorem \ref{th:cover} yields that $[(1_{P(S)}, 1, 1_{P(S)})]\varphi = 1_{P(S)}$. It follows that every proper Ehresmann monoid (in the signature $(2,1,1)$) has a proper cover which is a monoid and the covering morphism preserves the identity element. Proof of Theorem~\ref{th:cover} holds true also if we consider Ehresmann monoids in the extended signature $(2,1,1,0)$, so we conclude that every proper Ehresmann monoid (in the signature $(2,1,1,0)$) has a proper cover.
\end{remark}
\section{Connection with the work \cite{BGG15}} \label{s:connection} 

Suppose that $P(S)$ has the maximum element and consider the covering proper Ehresmann monoid \mbox{$P(S)\rtimes_{\mathcal{G}} X^*$ from the proof of Theorem~\ref{th:cover} in the signature $(2,1,1,0)$.} The construction of \mbox{$P(S)\rtimes_{\mathcal{G}} X^*$} implies that it is generated by elements $[(\overline{x}^+, x, \overline{x}^*)]$, where $x\in X$, and projections. In addition, the monoid homomorphism $\varphi\colon X^*\to P(S)\rtimes_{\mathcal{G}} X^*$, induced by the map $x\mapsto [(\overline{x}^+, x, \overline{x}^*)]$, is injective, because for any $v\in X^*$ we have $v={\mathbf l}(v\varphi)$ (note that $1\varphi = [(1_{P(S)},1,1_{P(S)}]$). It follows that $P(S)\rtimes_{\mathcal{G}} X^*$ is $X^*$-generated in the sense of \cite{BGG15}. Moreover, since different elements of $X^*\varphi$ have different labels, it follows that the congruence $\sigma$ separates $X^*$. Hence $P(S)\rtimes_{\mathcal{G}} X^*$ is strongly $X^*$-proper in the sense of \cite{BGG15}.

We now briefly recall the definition of the construction of the Ehresmann monoid ${\mathcal P}(T,Y)$ by Branco, Gomes and Gould from \cite{BGG15}. Let $T$ be a monoid, $Y$ a semilattice with a top element $1_Y$ and suppose that $T$ acts on $Y$ by order-preserving maps from the left and from the right subject to certain compatibility conditions. Denote the right action by $\circ$ and the left action by $\cdot$. These actions, together with the natural action of $Y$ on itself by the multiplication, extend to the right and the left actions, also denoted by $\circ$ and $\cdot$, of the semigroup free product $T*Y$ on $Y$. For any $u\in T*Y$ denote $u^+ = u\cdot 1_Y$ and $u^* = 1_Y\circ u$. Then one defines ${\mathcal P}(T,Y)$ as the quotient of $T*Y$ by a congruence $\sim$ which is the minimum congruence with the property that the operations $*$ and $+$ can be pushed down to the quotient $T*Y/\sim$ and this quotient can be endowed with the structure of an Ehresmann monoid. For more details, see \cite{BGG15}.

Let us show that the covering Ehresmann semigroup $P(S)\rtimes_{\mathcal{G}} X^*$ from the proof of Theorem \ref{th:cover} is isomorphic to the semigroup ${\mathcal P}(X^*, P(S))$ from \cite[Section 5]{BGG15}. We argue similarly to the last paragraph of the proof of Theorem \ref{th:cover}. Firstly, since ${\mathcal P}(X^*, P(S))$ is $X^*$-generated (in the sense of \cite{BGG15}), every element of $s\in {\mathcal P}(X^*, P(S))$ which is not a projection can be written as a product $s=(f_0x_1f_1) (f_1x_2f_2) \cdots (f_{n-1}x_nf_n)$ where $f_i\in P(S)$ and $x_i\in X$. Furthermore, applying Lemma \ref{lem:matching3}, this can be written as $s=(e_0x_1e_1) (e_1x_2e_2) \cdots (e_{n-1}x_ne_n) = e_0x_1e_1\cdots e_{n-1}x_ne_n$ where for all $i=1,\dots n:$ $(e_{i-1}x_ie_i)^+=e_{i-1}$, $(e_{i-1}x_ie_i)^*=e_i$. It is natural to say that the elements $e$, where $e\in P(S)$, and $e_0x_1e_1\cdots e_{n-1}x_ne_n$ are in the {\em canonical form}. 
It follows from \cite[Theorem 5.2]{BGG15} that the map ${\mathcal P}(X^*, P(S)) \to P(S)\rtimes_{\mathcal{G}} X^*$ given on the elements in the canonical form by $e\mapsto ([e,1,e)]$, where $e\in P(S)$, and
$$e_0x_1e_1\cdots e_{n-1}x_ne_n \mapsto [(e_0, x_1,e_1,\dots, e_{n-1}, x_n, e_n)],$$ 
is a $(2,1,1, 0)$-morphism, and it is obvious that it is a bijection. Thus it is a $(2,1,1,0)$-isomorphism of ${\mathcal P}(X^*, P(S))$ and $P(S)\rtimes_{\mathcal{G}} X^*$. In particular, the canonical form of $s\in   {\mathcal P}(X^*, P(S))$ is well defined.

Recall that the {\em free Ehresmann monoid} $FEM(X)$ and the {\em free Ehresmann semigroup} on the set $X$ are defined as the free objects in the varieties of $X$-generated Ehresmann monoids and of $X$-generated Ehresmann semigroups. Elegant combinatorial models for $FEM(X)$ and $FES(X)$ were proposed in \cite{Kam11} by Kambites. We note that $FEM(X)$ can be obtained from $FES(X)$ by adjoining of an external identity element.

We arrive at the following statement.

\begin{proposition}\label{prop:free_proper} The free Ehresmann monoid $FEM(X)$ and the free Ehresmnann semigroup $FES(X)$ are proper.
\end{proposition}

\begin{proof} Since $FEM(X)$ is $(2,1,1,0)$-generated by $X$, it has a proper cover $P(FEM(X))\rtimes_{\mathcal G} X^*$ constructed in the proof of Theorem \ref{th:cover}. By the above argument we have that $P(FES(X))\rtimes_{\mathcal G} X^*$ is isomorphic to ${\mathcal P}(X^*, P(FEM(X)))$, which, by \cite[Theorem 6.1]{BGG15}, is isomorphic to $FEM(X)$. The constructions imply that the covering morphism from \mbox{$P(FEM(X))\rtimes_{\mathcal G} X^*$} to $FEM(X)$ is in fact a $(2,1,1,0)$-isomorphism.
Detaching the external identity element in both $P(FEM(X))\rtimes_{\mathcal G} X^*$ and $FEM(X)$, we obtain the needed statement for $FES(X)$.
\end{proof}

\section{Special cases}\label{s:multiactions}
In this section we define partial multiactions of a monoid $T$ on a set $X$ and show that they are in a bijection with premorphisms $T\to {\mathcal B}(X)$. We then define partial multiactions of monoids on semilattices with compatible restrictions and corestrictions and show that these are a special case of labelled directed graphs with compatible restrictions and corestrictions. This leads to a structure result for proper left restriction (or proper right restriction) Ehresmann semigroups which generalizes the known structure result for proper restriction semigroups in terms of partial actions.

\subsection{Partial multiactions of monoids on sets}   
\begin{definition}\label{def:part_multiaction} (Partial multiactions) 
Let $T$ be a monoid and $X$ a non-empty set. Let ${\mathcal G}$ be a labelled directed graph with vertex set $X$ and edges labelled by elements of $T$ (see Definition \ref{def:labelled_graph_sets}). Suppose that for each $t\in T$ there is an edge in ${\mathcal G}$ labelled by $t$
and the following condition holds:
\begin{enumerate}
\item[(PM)] If $(x,t,y), (y,s,z)\in {\mathrm{E}}({\mathcal G})$ then $(x,ts,z)\in {\mathrm{E}}({\mathcal G})$.
\end{enumerate}
We define the edges $(x,t,y), (u,s,v)\in {\mathrm{E}}({\mathcal G})$ to be {\em composable} if $y=u$ in which case we put $(x,t,y)(y,s,z)=(x,ts,z)$. It is easy to verify this makes ${\mathcal G}$ a category with objects $X$ and arrows ${\mathrm{E}}({\mathcal G})$. We call it a {\em partial multiaction} of $T$ on $X$. 
\end{definition}

The identity arrow  of ${\mathcal G}$ at $x$ is $(x,1,x)$. We have ${\mathbf d}(x,t,y) = x$ and ${\mathbf r}(x,t,y) = y$.

\begin{remark} Recall \cite{L11} that the {\em Cauchy completion} of a semigroup $S$ is the category $$C(S) = \{(e, s, f )\in E(S) \times S \times E(S)\colon esf = s\}$$
with the composition rule $(e,s,f)(f,t,g)=(e, st,g)$. If we extend the definition of a partial multiaction from monoids to semigroups by requiring that (PM) holds, then $C(S)$ 
becomes an example of  a partial multiaction of the semigroup $S$ with $X=E(S)$.
\end{remark}

Let us show that the notion of a partial multiaction of $T$ on $X$ subsumes those of both left and right partial actions of $T$ on $X$. Recall that a {\em right partial action of} $T$ on $X$ is a partially defined map $X\times T \to X$, $(x,t)\mapsto x\cdot t$, such that for every $t\in T$ there is $x\in X$ such that $x\cdot t$ is defined and:
\begin{enumerate} [(i)]
\item if $x\cdot s$ and $(x\cdot s)\cdot t$ are defined then $x\cdot st$ is defined and $(x\cdot s)\cdot t=x\cdot st$.
\item for all $x\in X$ we have that $x\cdot 1$ is defined and $x\cdot 1 =x$.
\end{enumerate} 
{\em Left partial actions} of $T$ on $X$ are defined dually.

Suppose that the (partially defined) assignment $(x,t)\mapsto x\cdot t$ defines a right partial action of $T$ on $X$ and define ${\mathcal G}$ to be the labelled directed graph with vertex set $X$, edges labelled by elements of $T$ and ${\mathrm{E}}({\mathcal G}) = \{(x,t,x\cdot t) \colon x\cdot t \text{ is defined}\}$. Then ${\mathcal G}$ is a partial multiaction of $T$ on $X$ and also satisfies the following condition:
\begin{enumerate}
\item[(LD)] (Left determinism) For all $t\in T$ and $x\in X$ there exists at most one $y\in X$ such that $(x,t,y)\in {\mathrm{E}}({\mathcal G})$.
\end{enumerate} 

In this way right partial actions of $T$ on $X$ are in a bijective correspondence with partial multiactions of $T$ on $X$ which satisfy condition (LD). 
Dually, given a left partial action $(t,x)\mapsto t\cdot x$ of $T$ on $X$ we define ${\mathcal G}$ to be the labelled directed graph with vertex set $X$, edges labelled by elements of $T$ and ${\mathrm{E}}({\mathcal G}) = \{(t\cdot x,t,x) \colon t\cdot x \text{ is defined}\}$. Then ${\mathcal G}$ is a partial multiaction of $T$ on $X$ and also satisfies the following condition:

\begin{enumerate}
\item[(RD)] (Right determinism) For all $t\in T$ and $y\in X$ there exists at most one $x\in X$ such that $(x,t,y)\in {\mathrm{E}}({\mathcal G})$.
\end{enumerate} 

Hence left partial acitons of $T$ on $X$ are in a bijective correspondence with partial multiactions of $T$ on $X$ which satisfy condition (RD).

\subsection{Partial multiactions and premorphisms} We now introduce the notion of a {\em premorphism} from a monoid $T$ to the monoid ${\mathcal B}(X)$ and relate it with the notion of a  partial multiaction of $T$ on $X$. 

\begin{definition}[Premorphisms]
Let $T$ be a monoid and $X$ a set. A map $\varphi\colon T\to {\mathcal B}(X)$, $t\mapsto \varphi_t$,  will be called a {\em  premorphism} provided that for all $t\in T$, $\varphi_t\neq\varnothing$, and
\begin{enumerate}
\item[(Prem1)]  $id_X \subseteq \varphi_1$,
\item[(Prem2)]  for all $s,t\in T, \varphi_s\varphi_t\subseteq \varphi_{st}$.
\end{enumerate}
\end{definition}

\begin{proposition}\label{prop:terminology} \mbox{}
\begin{enumerate}
\item \label{i:22m1} Let ${\mathcal G}$ be a partial multiaction of $T$ on $X$. For every $t\in T$ define $\varphi_t = \{(x,y)\in X\times X\colon (x,t,y)\in {\mathrm{E}}({\mathcal G})\}$. Then $\varphi\colon T\to {\mathcal B}(X)$, $t\mapsto \varphi_t$, is a premorphism.
\item \label{i:22m2} Let $\varphi\colon T\to {\mathcal B}(X)$ be a premorphism and define ${\mathcal G}$ to be the category with objects $X$ and arrows ${\mathrm{E}}({\mathcal G}) = \{(x,t,y)\colon t\in T, (x,y)\in \varphi_t\}$. Then ${\mathcal G}$ is a partial multiaction of $T$ on $X$.
\item \label{i:22m3} A partial multiaction ${\mathcal G}$ of $T$ on $X$ satisfies condition (LD) (respectively condition (RD)) if and only if $\varphi(T)\subseteq {\mathcal{PT}}(X)$ (respectively $\varphi(T)\subseteq {\mathcal{PT}}^c(X)$) where $\varphi\colon T\to {\mathcal B}(X)$ is the premoprhism defined in part (\ref{i:22m1}).
\item \label{i:22m4} A partial multiaction ${\mathcal G}$ of $T$ on $X$ satisfies both (LD) and (RD)  if and only if $\varphi(T) \subseteq {\mathcal I}(X)$ where $\varphi\colon T\to {\mathcal B}(X)$ is the premoprhism defined in part (\ref{i:22m1}).
\end{enumerate}
\end{proposition} 

\begin{proof} Parts (1), (2), and (3) are direct consequences of  the definitions. Part (4) follows from (3) and ${\mathcal{PT}}(X)\cap {\mathcal{PT}}^c(X) = {\mathcal I}(X)$.
\end{proof}

We call a partial multiaction ${\mathcal{G}}$ and the premorphism $\varphi$ defined in Proposition \ref{prop:terminology}(\ref{i:22m1}) {\em attached} to each other.

\begin{remark} Suppose that a partial multiaction ${\mathcal G}$ satisfies (LD) and (RD) and let $\varphi\colon T\to {\mathcal B}(X)$ be its attached premorphism. Then for each $(x,t,y)\in {\mathrm{E}}({\mathcal G})$ we have $y=x\varphi_t$ or, equivalently, $x=y\varphi_t^{-1}$ where $\varphi_t^{-1}$ is the reverse relation to $\varphi_t$ and coincides with the inverse of $\varphi_t$ in ${\mathcal I}(X)$. Therefore, ${\mathrm{E}}({\mathcal G}) = \{(x,t,x\varphi_t)\colon t\in T, x\in \dom(\varphi_t)\} = \{(y\varphi_t^{-1},t,y)\colon t\in T, y\in \ran(\varphi_t)\}$. Also, ${\mathrm{E}}({\mathcal G}) = \{(\varphi_ty, t, y)\colon t\in T, y\in \ran(\varphi_t)\} = \{(x,t,\varphi_t^{-1}x)\colon t\in T, x\in\dom(\varphi_t)\}$.
\end{remark}

Observe that if ${\mathcal G}$ satisfies condition (LD) then for its attached premorphism $\varphi\colon T\to {\mathcal B}(X)$ condition (Prem1) reduces to $\varphi_1={\mathrm{id}}_X$. In addition, we have $\varphi_s\subseteq \varphi_t$ if and only $\varphi_s\leq \varphi_t$, where $\leq$ is the natural partial order on ${\mathcal{PT}}(X)$. This leads to the following corollary of Proposition \ref{prop:terminology}.

\begin{corollary}\label{cor:multiact_prem}
There is a bijective correspondence between partial multiactions of $T$ on $X$ and premorphisms $T\to {\mathcal B}(X)$. This correspondence subsumes the bijective correspondences between:
\begin{enumerate}
\item \label{i:23m1a} right (respectively left) partial actions of $T$ on $X$ and premorphisms $T\to {\mathcal{PT}}(X)$ (respectively $T\to {\mathcal{PT}}^c(X)$), see \cite[Proposition~2.8]{H07},
\item \label{i:23m2a} right (or, equivalently, left) partial actions of $T$ on $X$ by partial bijections and premorphisms $T\to {\mathcal I}(X)$, see \cite{K15}.
\end{enumerate}
\end{corollary}

\subsection{Partial multiactions of monoids on semilattices with compatible restrictions and corestrictions}  Let $T$ be a monoid, $E$ a semilattice and ${\mathcal G}$ a labelled directed graph with compatible restrictions and corestrictions (see Definition  \ref{def:labelled_graph}). In this subsection we additionally suppose that  ${\mathcal G}$ is a partial multiaction of $T$ on $E$ (see Definition \ref{def:part_multiaction}). This means that if $(e,t,f), (f,s,g)\in {\mathrm{E}}({\mathcal G})$ then $(e,ts,g)\in {\mathrm{E}}({\mathcal G})$. We then call the category ${\mathcal G}$ a {\em partial multiaction of} $T$ {\em on} $E$ {\em with compatible restrictions and corestrictions.} Note that every $\equiv$-class of the path semicategory ${\mathcal C}(\mathcal G)$ contains a unique path of length $1$, which is an edge of ${\mathcal G}$. It follows that the category ${{\widetilde{\mathcal C}}}({\mathcal G})$ is isomorphic to ${\mathcal G}$ via the identity map on objects and the map $[(e,t,f)]\mapsto (e,t,f)$ on the arrows. From now on we identify ${{\widetilde{\mathcal C}}}({\mathcal G})$ with ${\mathcal G}$. In particular, the underlying set of the Ehresmann semigroup $E\rtimes_{\mathcal G} T$ coincides with ${\mathrm{E}}({\mathcal G})$.

For a strictly proper Ehresmann semigroup $S$ we put ${\mathcal G}_S$ to be the underlying graph ${\mathcal G}_{S,S}$ of $S$ with respect to the proper generating ideal $S$. Applying Theorem \ref{th:construction}\eqref{i:proper} we obtain the following statement.

\begin{proposition}\label{prop:strictly}\mbox{}
\begin{enumerate}
\item\label{i:pr41}
 Let ${\mathcal G}$ be a partial multiaction of a monoid $T$ on a semilattice $E$ with compatible restrictions and corestrictions and assume that for all $c, d\in E\rtimes_{\mathcal G} T$: $c \mathrel{\sigma} d$ holds if and only if ${\mathbf l}(c) = {\mathbf l}(d)$. Then $E\rtimes_{\mathcal G} T$ is strictly proper and $(E\rtimes_{\mathcal G} T)/\sigma\simeq T$ via the map $[(e,s,f)]_{\sigma}\mapsto s$.
 \item\label{i:pr42} Let $S$ be a strictly proper Ehresmann semigroup. Then the underlying labelled directed graph  ${\mathcal G}_S$ is a partial multiaction of $S/\sigma$ on $P(S)$.
 \end{enumerate}
\end{proposition}

\subsection{Strictly proper Ehresmann semigroups arising from deterministic partial multiactions}  Let ${\mathcal G}$ be a partial multiaction of a monoid $T$ on a semilattice $E$ with compatible restrictions and corestrictions.

\begin{lemma} \label{lem:proper} Let ${\mathcal G}$ satisfy (LD) or (RD). Then:
\begin{enumerate}
\item\label{lem4:1} $(e,s,f) \mathrel{\sigma} (g,t,h)$ if and only if $s=t$;
\item \label{lem4:2} $E\rtimes_{\mathcal G} T$ is strictly proper and $(E\rtimes_{\mathcal G} T)/\sigma \simeq T$ via the map $[(e,s,f)]_{\sigma}\mapsto s$.
\end{enumerate}
 \end{lemma}
 
\begin{proof} 
(1) In view of Theorem \ref{th:construction}(\ref{i:sigma}) we only need to prove that $s=t$ implies that $(e,s,f) \mathrel{\sigma} (g,t,h)$. Consider the case where ${\mathcal G}$ satisfies condition (LD), the other case is dual. So let $s\in T$ and show that $(e,s,f) \mathrel{\sigma} (g,s,h)$.  Applying \eqref{eq:def_product} and (CR5b), we have $(eg,1,eg)\cdot (e,s,f) =
(eg,1,eg)|_{eg}\, _{eg}|(e,s,f) = {}_{eg}|(e,s,f)$.
Since $\sigma$ identifies all the projections, 
$_{eg}|(e,s,f)= (eg,1,eg)\cdot (e,s,f) \mathrel{\sigma} (e,1,e)\cdot (e,s,f) = (e,s,f)$ and similarly  $_{eg}|(g,s,h) \mathrel{\sigma} (g,s,h)$. Note that ${\bf d}(_{eg}|(e,s,f)) = {\bf d}(_{eg}|(g,s,h)) = eg$. Condition (LD) implies that \mbox{$_{eg}|(e,s,f)  \space{} =  {}_{eg}|(g,s,h)$.} This yields that  $(e,s,f)\mathrel{\sigma} (g,s,h)$, as desired.

(2) This follows from part (1) and  Proposition \ref{prop:strictly}\eqref{i:pr41}.
\end{proof}
 
We now demonstrate that proper left (respectively right) restriction Ehresmann semigroups arise from partial multiactions of monoids on semilattices with compatible restrictions and corestrictions that satisfy condition (LD) (respectively condition (RD)).

\begin{proposition}\label{prop:left_restr}
If ${\mathcal G}$ satisfies condition (LD) (respectively condition (RD), or conditions (LD) and (RD)) then
 the Ehresmann semigroup $E\rtimes_{\mathcal G} T$ is a proper left restriction semigroup (respectively proper right restriction semigroup, or proper restriction semigroup) with $(E\rtimes_{\mathcal G} T)/\sigma \simeq T$ via the map $[(e,s,f)]_{\sigma}\mapsto s$.
\end{proposition}

\begin{proof} We consider only the case of condition (LD), the case with (RD) is dual, and the third case follows from the first two. From Lemma \ref{lem:proper} it follows that the Ehresmann semigroup $E\rtimes_{\mathcal G} T$ is strictly proper with $(E\rtimes_{\mathcal G} T)/\sigma \simeq  T$ via the map $[(e,s,f)]_{\sigma}\mapsto s$. We verify that $E\rtimes_{\mathcal G} T$ satisfies the left ample identity \eqref{eq:mov_proj}. Let $a = (e,s,h)$ and $f=(g,1,g)$. Then by \eqref{eq:def_product} and (R5) we have $a\cdot f = (e,s,h)\cdot (g,1,g) = (e,s,h)|_{hg} \,  {}_{hg}|(g,1,g)= (e',s,hg) (hg,1,hg)=(e',s,hg)$ 
where $e' = {\mathbf d}((e,s,h)|_{hg})$. Hence $(a\cdot f)^+ = (e',1,e')$ and 
$$
(a\cdot f)^+\cdot a = (e',1,e')\cdot (e,s,h) = \space{} _{e'}|(e,s,h)=(e', s, h')
$$
where $h'\leq h$ and for the last equality we applied (R1). Since $(e',s, hg), (e',s, h') \in {\mathrm{E}}({\mathcal G})$, condition (LD) yields $h'=hg$. This implies that $a\cdot f = (a\cdot f)^+\cdot a$, as needed. We finally show that $E\rtimes_{\mathcal G} T$ is a proper left restriction semigroup. Assume that $a,b\in E\rtimes_{\mathcal G} T$ are such that $a^+=b^+$ and $a\mathrel{\sigma} b$. This means that $a=(e,s,f)$ and $b=(e,s,g)$ for some $e,f,g\in E$ and $s\in T$. Condition (LD) now yields that $f=g$, that is, $a=b$.
\end{proof}

 The following statement shows that strictly proper Ehresmann semigroups generalize Ehresmann semigroups which are proper left (or right) restriction and they also generalize proper restriction semigroups. It follows directly from the definition of properness.
 
 \begin{lemma}\label{lem:proper_Ehr_left_restr}
 Let $S$ be an Ehresmann semigroup, which is a proper left (or right) restriction semigroup. Then $S$ is a strictly proper Ehresmann semigroup. Consequently, if $S$ is a proper restriction semigroup then it is a strictly proper Ehresmann semigroup.
 \end{lemma}
 
We now determine when an Ehresmann semigroup is proper left restriction (respectively proper right restriction, or proper restriction).
 
\begin{theorem} \label{th:proper_restriction} Let $S$ be an Ehresmann semigroup. Then $S$ is proper left restriction (respectively proper right restriction, or  proper restriction) if and only if $S$ is strictly proper Ehresmann and its underlying partial multiaction ${\mathcal G}_S$ satisfies condition (LD) (respectively (RD), or both (LD) and (RD)).
\end{theorem}

\begin{proof} We consider only the case of a left restriction semigroup, the case of a right restriction semigroup  being dual and the third case following from the first two. Suppose that $S$ is a proper left restriction semigroup. Then it is a strictly proper Ehresmann semigroup and let $a=(e,s,f), b=(e,s,g)\in P(S)\rtimes_{{\mathcal G}_S} S/\sigma$. Then $a^+ = b^+$ and $a \mathrel{\sigma} b$ in $P(S)\rtimes_{{\mathcal G}_S} S/\sigma$. By Theorem~\ref{th:main} $P(S)\rtimes_{{\mathcal G}_S} S/\sigma$ is isomorphic to $S$, hence it is a proper left restriction semigroup, so $a=b$. It follows that ${\mathcal G}_S$ satisfies condition (LD). 

For the reverse direction, suppose that $S$ is strictly proper Ehresmann and its underlying partial multiaction ${\mathcal G}_S$ satisfies condition (LD). Proposition \ref{prop:left_restr} implies that $P(S)\rtimes_{{\mathcal G}_S} S/\sigma$ is a proper left restriction semigroup. In view of Theorem \ref{th:main}, the statement follows.
\end{proof}

As a corollary we obtain a structure result for Ehresmann semigroups which are proper left restriction (or proper right restriction) and of proper restriction semigroups.

\begin{corollary}\label{cor:special}\mbox{}
\begin{enumerate}
\item  Let $S$ be an Ehresmann semigroup which is proper left restriction (respectively proper right restriction). Then it is a strictly proper Ehresmann semigroup, its underlying partial multiaction ${\mathcal G}_S$ satisfies condition (LD) (respectively condition (RD)) and $S$ is isomorphic to $P(S)\rtimes_{{\mathcal G}_S} S/\sigma$.
\item Let $S$ be a proper restriction semigroup. Then it is a strictly proper Ehresmann semigroup, its underlying partial multiaction ${\mathcal G}_S$ satisfies conditions (LD) and (RD) and $S$ is isomorphic to $P(S)\rtimes_{{\mathcal G}_S} S/\sigma$.
\end{enumerate}
\end{corollary}

\subsection{The structure of proper restriction semigroups}\label{subs:proper_restr}  In this subsection we show that  the known result on the structure of proper restriction semigroups \cite{CG12, K15} is equivalent to Corollary \ref{cor:special}(2), thus it can be recovered as a special case of Theorem \ref{th:main}. 

For a semilattice $E$ let $\Sigma(E)$ be the inverse semigroup of all order isomorphisms between order ideals of $E$ \cite[VI.7.1]{Petrich}.

\begin{proposition}\label{prop:dom_ran_ideals} Let ${\mathcal G}$ be a partial multiaction of a monoid $T$ on a semilattice $E$ with compatible restrictions and corestrictions and let $\varphi\colon T\to {\mathcal B}(E)$ be its attached premorphism.
\begin{enumerate}
\item If ${\mathcal G}$ satisfies condition (LD) (respectively condition (RD)) 
then, for all $t\in T$, ${\mathrm{dom}}(\varphi_t)$ (respectively ${\mathrm{ran}}(\varphi_t)$) is an order ideal of $E$ and the map $\varphi_t$ (respectively $\varphi_t^{-1}$) is order-preserving.
\item If ${\mathcal G}$ satisfies both of the conditions (LD) and (RD) then for all $t\in T$ we have that $\varphi_t\in\Sigma(E)$. 
\end{enumerate}
\end{proposition}

\begin{proof}
(1) We consider the case where ${\mathcal G}$ satisfies (LD), the other case being dual. Let $e\in {\mathrm{dom}}(\varphi_t)$ and $f\leq e$. Then $_f|(e,t,e\varphi_t) = (f,t,g)$ where $g\leq e\varphi_t$ by (R1) which implies that $f\in\dom(\varphi_t)$, thus $\dom(\varphi_t)$ is an order ideal. Furthermore, it follows from (LD) that $g=f\varphi_t$ so that $_f|(e,t,e\varphi_t) = (f,t,f\varphi_t)$. It now follows by (R1) that $f\varphi_t\leq e\varphi_t$. Hence $\varphi_t$ is order-preserving.

(2) Let ${\mathcal G}$ satisfy both of the conditions (LD) and (RD). Then $\dom(\varphi_t)$ and $\ran(\varphi_t)$ are order-ideals of $E$, in addition $\varphi_t$ is an injective map and both $\varphi_t$ and $\varphi_t^{-1}$ are order-preserving. This means that $\varphi_t\in \Sigma(E)$, as needed.
\end{proof}

Let $S$ be a proper restriction semigroup. Denote $T=S/\sigma$. Let $\varphi\colon T\to {\mathcal B}(P(S))$ be the premorphism attached to the partial multiaction ${\mathcal G}_S$. It follows from Theorem \ref{th:proper_restriction} that $\varphi(T)\subseteq \Sigma(P(S))$. Let $t\in T$. We have
$$
\dom(\varphi_t) = \{e\in P(S)\colon \exists a\in S \text{ such that } e=a^+ \text{ and } [a]_{\sigma}=t\}.
$$
If $e\in\dom(\varphi_t)$ then $e\varphi_t=a^*$
and $a^*\varphi_t^{-1}=a^+$ where $a\in S$ is such that $[a]_{\sigma}=t$ and $a^+=e$. Since $\varphi e = e\varphi^{-1}$ for all $\varphi\in {\mathcal{I}}(E)$, the operations on $P(S)\rtimes_{{\mathcal{G}}_S} T$ can be written as follows:
\begin{align*}
(e,s,e\varphi_s)\cdot (f,t,f\varphi_t)  & =  ((e\varphi_s\wedge f)\varphi_s^{-1},s,e\varphi_s\wedge f)\cdot(e\varphi_s\wedge f, t, (e\varphi_s\wedge f)\varphi_t)\\
& = ((e\varphi_s\wedge f)\varphi_s^{-1}, st, (e\varphi_s\wedge f)\varphi_t),
\end{align*}
$$
(e,s,e\varphi_s)^+=(e,1,e), \,\,\, (e,s,e\varphi_s)^* = (e\varphi_s,1,e\varphi_s).
$$

Since the third component of a triple $(e,s,e\varphi_s)$ is determined by the first two components, $(e,s,e\varphi_s)$ is determined by the pair $(e,s)$. The set ${\mathcal A}$ of all such pairs is in a bijection with the underlying set of $P(S)\rtimes_{{\mathcal{G}}_S} T$. The operations on  $P(S)\rtimes_{{\mathcal{G}}_S} T$ are translated to operations on ${\mathcal A}$ as follows:
\begin{equation}\label{eq:A}
(e,s)\cdot (f,t) = ((e\varphi_s\wedge f)\varphi_s^{-1}, st),\,\, (e,s)^+ = (e,1), \,\, (e,s)^* = (e\varphi_s,1). 
\end{equation}
On the other hand, let $\varphi$ be a partial action of $T$ on $E$ by partial bijections between order ideals such that  $\dom(\varphi_t)\neq \varnothing$ for all $t\in T$. The graph assigned to $\varphi$ has vertex set $E$ and edges $(e,t,f)$ where $t\in T$, $e\in\dom(\varphi_t)$ and $f=e\varphi_t$. We define the restriction of $(e,t,f)$ to $g\leq e$ by $_g|(e,t,f) = (g,t,g\varphi_t)$ and the corestriction of $(e,t,f)$ to $h\leq f$ by $(e,t,f)|_h = (h\varphi_t^{-1},t,h)$. It is routine to verify that then ${\mathcal{G}}$ is a partial multiaction with compatible restrictions and corestrictions which satisfies conditions (LD) and (RD). We have rediscovered the structure result on proper restriction semigroups, as it is formulated in \cite[Theorem~3]{K15}.

We conclude the paper with the following open questions.

\begin{question}\label{q:1} Is the free Ehresmann semigroup $FES(X)$ strictly proper? 
\end{question}

In other words, is it true that every element of $a \in FES(X)$ is uniquely determined by $a^*$, $a^+$ and $[a]_{\sigma}$? It is known \cite{BGG15} and easy to see that, unlike what happens in proper restriction semigroups,  an element $a\in FES(X)$ is not in general uniquely determined only by $a^+$ (or by $a^*$) and $[a]_{\sigma}$. For example, if $X=\{x,y\}$, the elements $xy^+$ and $(xy)^+x$ are different (because their underlying Kambites trees \cite{Kam11} are different) and we have that $(xy^+)^+ = ((xy)^+x)^+$ and $xy^+ \mathrel{\sigma} (xy)^+x$. However, $(xy^+)^* \neq ((xy)^+x)^*$.

\begin{question} Does every Ehresmann semigroup have a strictly proper cover?
\end{question}

\section*{Acknowlegements} The authors thank the anonymous referee for a very careful reading of the paper and a number of suggestions which have improved the exposition.


\end{document}